\documentclass[11pt, reqno]{amsart}
\RequirePackage[OT1]{fontenc}
\usepackage[colorlinks=true, urlcolor=blue, linkcolor=blue, citecolor=blue, pdftex]{hyperref}
\usepackage{csquotes}
\usepackage[utf8x]{inputenc}
\usepackage{ucs}
\usepackage[english]{babel}
\usepackage[T1]{fontenc}
\usepackage{amsmath}
\usepackage{amsfonts}
\usepackage{amssymb}
\usepackage{enumerate}

\usepackage{algorithmicx}
\usepackage{algpseudocode}
\usepackage{algorithm}

\usepackage{tikz}

\usepackage{titlesec}

\titleformat{\section}[block]{\normalfont\bfseries\filcenter}{\itshape\thesection}{0.9em}{}
\titleformat{\subsection}[block]{\normalfont\bfseries}{\itshape\thesubsection}{0.6em}{}
\titleformat{\subsubsection}[block]{\normalfont\bfseries}{\itshape\thesubsubsection}{0.3em}{}

\newcommand{\bgamma}{\bar{\gamma}}

\newcommand{\bbu}{\mathbf{u}}

\newcommand{\tp}{\tilde{p}}
\newcommand{\hp}{\hat{p}}

\newcommand{\tX}{\tilde{X}}

\newcommand{\E}{\mathbb{E}}

\newcommand{\R}{\mathbb{R}}

\newcommand{\mP}{\mathcal{P}}

\newcommand{\mL}{\mathcal{L}}

\newcommand{\p}{\partial}

\renewcommand{\p}{\partial}

\newcommand{\la}{\langle}
\newcommand{\ra}{\rangle}

\renewcommand{\l}{\left}
\renewcommand{\r}{\right}
\renewcommand{\d}{{\rm d}}

\numberwithin{equation}{section}

\newtheorem{theoreme}{Theorem}[section]
\newtheorem{Proposition}[theoreme]{Proposition}

\newtheorem{lemme}[theoreme]{Lemma}

\newtheorem{remarque}{Remark}

\newtheorem{claim}[theoreme]{Claim}
\begin{document}
\title[Strong well-posedness of MKV-SDE with Hölder drift]{Strong well-posedness of McKean-Vlasov stochastic differential equation with Hölder drift}

\author{P.E. Chaudru de Raynal}
\address{UNIVERSITE SAVOIE MONT BLANC, LAMA. }
\email[]{pe.deraynal@univ-savoie.fr}


\begin{abstract}
In this paper, we prove pathwise uniqueness for stochastic systems of McKean-Vlasov type with singular drift, even in the measure argument, and uniformly non-degenerate Lipschitz diffusion matrix. 

Our proof is based on Zvonkin's transformation \cite{zvonkin_transformation_1974} and so on the regularization properties of the associated PDE, which is stated on the space $[0,T]\times \R^d\times \mathcal{P}_2(\R^d)$, where $T$ is a positive number, $d$ denotes the dimension equation and $\mathcal{P}_2(\R^d)$ is the space of probability measures on $\R^d$ with finite second order moment. Especially, a smoothing effect in the measure direction is exhibited. Our approach is based on a parametrix expansion of the transition density of the McKean-Vlasov process.
\end{abstract}

\maketitle

\section{Introduction}
Let $\mathcal{M}_{d}(\mathbb{R})$ be the set of $d\times d$ matrices with real coefficients and $\mathcal{P}_{2}(\R^d)$ be the space of probability measures $\nu$ on $\R^d$ such that $\int x^2 \d \nu(x) <+\infty$. For any random variable $Z$, let us denote by $[Z]$ its law and for any measurable function $\varphi$ let us write $\int \varphi \d \nu$ with its dual notation: $\langle \varphi, \nu\rangle$.

For a positive number $T$, for given measurable functions $b,\sigma: [0,T] \times \mathbb{R}^d \times \mathbb{R}  \to \mathbb{R}^d \times \mathcal{M}_{d}(\mathbb{R})$, $\varphi_i : \R^d\to \R$, $i$ in $\{1,2\}$ and for $(B_{t}, t\geq 0)$ a standard $d$-dimensional Brownian motion defined on a filtered probability space $(\Omega, \mathcal{F}, \mathbb{P}, (\mathcal{F}_t)_{ t \geq 0})$, we consider, for $t<s$ in $[0,T]^2$ and $\mu$ in $\mathcal{P}_2(\R^d)$, the non-linear (in a McKean-Vlasov sense) system
\begin{equation}\label{ttarget1}
X_s^{t,\mu} = X_t + \int_t^s b(r,X_r^{t,\mu},\langle \varphi_1, [X^{t,\mu}_r]\rangle ) \d r + \int_t^s\sigma(r,X_r^{t,\mu},\langle \varphi_2, [X^{t,\mu}_r]\rangle) \d B_r,\quad X_t \sim \mu.
\end{equation}

This sort of equation arises as the limit of system of interacting players. This happens as follows. Suppose that we are given a large number of players with symmetric dynamic and whose positions depend on the positions of the other players in a mean field way. Then,  when the number of players tends to infinity, there is a propagation of chaos phenomenon so that the limit dynamic of each player does not depend on the positions of the others anymore, but only on their statistical distributions. This obviously comes from the law of large numbers. The resulting system is then of the form of \eqref{ttarget1} and is called non-linear, since the dynamic of the player depends on its own law. We refer to the notes of Sznitman's lecture at Saint-Flour \cite{sznitman_topics_1991} for an overview on the topic.\\

As done in \cite{sznitman_topics_1991}, the proof of strong existence and uniqueness (which means that the solutions are adapted to the filtration generated by the Brownian motion and are almost surely indistinguishable) for this equation relies on classical fixed point argument and so, on the Lipschitz property of the coefficients of the equation, the Lispchitz regularity being understood with respect to (w.r.t.) the Wasserstein metric in the case of the measure argument. 

In this work, we aim at proving the strong well-posedness of such a system when the diffusion matrix is Lipschitz w.r.t. both the space and measure arguments and uniformly non-degenerate but when the drift is only a bounded in space and Hölder continuous w.r.t. the measure function in the following sense: the drift $b$ is asking to be bounded in space and Lipschitz w.r.t. the third argument but the map $\varphi_1$ is supposed to be only Hölder-continuous. This means that when rewriting the drift of \eqref{ttarget1} as the map $B: [0,T]\times \R^d \times  \mathcal{P}_2(\R^d) \ni (t,x,\nu) \mapsto B(t,x,\nu) = b(t,x,\langle\varphi_1,\nu\rangle) \in \R^d$, the function $B$ is only assumed to be Hölder-continuous w.r.t. the measure argument (for the Wasserstein distance) and bounded in space.

To the best of our knowledge, this result is new and relies on regularization by noise phenomenon (see \cite{flandoli_random_2011} for a survey). This effect comes from the random perturbation of the equation, allowing stochastic differential system to be well-posed in a strong or a weak sense, under a larger set of assumptions than ordinary differential system. It thus deeply relies on the noise propagation through the system which is the reason why we suppose the matrix $\sigma\sigma^*$ to be uniformly non-degenerate. Our work then consists, on the one hand, to show that this phenomenon still holds in the McKean-Vlasov setting and, on the other hand, even occurs in the measure space. This last result is quite unexpectable at first sight, since the noise does not act in that direction. We nevertheless show that is is indeed the case, thanks to the structural assumption made on the dependence of the system w.r.t. the measure argument (\emph{i.e.} when the dependence upon the measure is of polynomial type).

To do so, we adapt to our framework the Zvonkin transformation \cite{zvonkin_transformation_1974}. This approach relies on smoothing properties of a well-chosen PDE associated to the system \eqref{ttarget1}. Here, the non-linearity (in a Mckean-Vlasov sense) leads to a particular class of PDE that can be seen as the linear version of the so called Master Equation coming from Mean Field Games theory introduced independently by Lasry and Lions \cite{lasry_jeux_2006,lasry_jeux_2006-1, lasry_mean_2007} and by Huang, Caines and Malhameé \cite{huang_large_2006}. This PDE has been recently studied from a probabilistic point of view in the independent works of Buckdahn, Li, Peng and Rainer in \cite{Buckdahn_Mean-field_2014} and of Crisan, Chassagneux and Delarue in \cite{chassagneux_probabilistic_2014}. Its main particularity comes from the fact that it is stated on the space $[0,T]\times \R^d\times \mathcal{P}_2(\R^d)$ so that it involves derivatives in the measure direction.

The smoothing properties of the associated PDE is then the crucial part of the proof. We investigate it under a larger set of assumptions: we let the diffusion matrix be only a Hölder-continuous function of the space and measure argument. The investigations are done by using  a Feynman-Kack representation of the solution of the PDE and then a parametrix expansion (see \cite{mckean_jr._curvature_1967}) of the transition density of the solution of \eqref{ttarget1}. This brings us to investigate for all $t<s$ in $[0,T]^2$ the regularity of $\mathcal{P}_2(\R^d) \ni \mu \mapsto \langle \phi, [X_s^{t,\mu}]\rangle \in \R$ for some Hölder continuous function $\phi : \R^d \to \R$ and where $X$ is a solution of \eqref{ttarget1}. Especially, as it has been announced above, a smoothing effect w.r.t. measure argument (\emph{i.e.} on the initial data ``$\mu$'') is exhibited. \\

Finally, let us emphasize that the parametrix approach relies on perturbation approach. This explains why, in comparison with the results obtained in the linear case (see \cite{zvonkin_transformation_1974,veretennikov_strong_1980}), the coefficients of \eqref{ttarget1} are assumed to be Hölder continuous and not only bounded functions of the space and measure arguments.

\subsection*{Organization of this paper.} Our paper is organized as follows: we present below our main assumptions and result. Then, we give in Section \ref{SoP} the mathematical background and our strategy of proof. Especially, we state in this section the PDE associated to \eqref{ttarget1}. Since the PDE is stated on the Cartesian space $[0,T]\times\R^d \times \mathcal{P}_2 (\R^d)$, we give some notions on differentiation of functions along a measure. Then, we establish the smoothing properties of the PDE  associated to \eqref{ttarget1},  which is a key result in the proof of our main Theorem.

Next, we investigate the smoothing properties of the PDE. Such an investigation is done under regularized framework. It is based on a parametrix expansion of the transition density of \eqref{ttarget1} and is presented as follows. In Section \ref{EotS} we give estimates on the transition density of \eqref{ttarget1} and on the mapping $v: \mathcal{P}_2(\R^d) \ni \mu \mapsto \langle \phi, [X_s^{t,\mu}]\rangle \in \R$ for some Hölder continuous function $\phi : \R^d \to \R$. This permits to estimate the solution of the PDE. Then, estimates on $v$ are proven in Section \ref{Eov} and estimates on the transition density of \eqref{ttarget1} are proven in Section \ref{Parametrix}. As said before, this last follows from a parametrix representation of the transition density of \eqref{ttarget1} and its estimation. Auxiliary results are given in Appendixes \ref{sec:prooflemmas} and \ref{APPC}.\\

\subsection*{Notations, assumptions and main result}

\textbf{Notations.} For any function $f : E\times F \times G \to \R^N$, we denote by $\p_1$ (resp. $\p_2$ and $\p_3$) the differentiation w.r.t. the first (resp. second and third) variable. When we add a subscript in the operator $\p_\cdot$, it stands for the variable on which the differentiation operator acts. We recall that the law of a random variable $X$ is denoted by $[X]$. The superscript ``$*$'' stands for the transpose, the canonical Euclidean inner product on $\R^d$ is denoted by ``$\cdot$''. We denote by $\mathcal{M}_d(\R)$ the set of $d\times d$ matrices with real coefficients and the trace of a matrix $M$ in $\mathcal{M}_d(\R)$ is denoted by ${\rm Tr}(M) = \sum_{j=1}^dM_{jj}$. We let $C, C', c, c',\tilde{C}, \tilde{C}',\ldots$ be some positive constants  depending only on known parameters in \textbf{(HE)} given below, that may change from line to line and from an equation to another and we add a subscript $T$ in the constant if it depends also on the length of the interval.\\

\textbf{Assumptions (HE).} We say that assumptions \textbf{(HE)} hold if the following assumptions are satisfied:
\begin{description}
\item[(HE1) regularity of the drift]  there exists a positive constant $C_b$ such that $||b||_{\infty} < C_b$. Moreover for all $(t,x)$ in $[0,T] \times \R^d$, the mapping $b(t,x,\cdot) : \R \ni w \mapsto b(t,x,w)$ is differentiable and $||\p_3 b||_{\infty} < C_b'$. Finally, the mapping $\varphi_1 : \R^d \ni x \mapsto \varphi_1(x)$ is supposed to be $\alpha_1$-Hölder for some $0<\alpha_1\leq 1$.\\

\item[(HE2) regularity of the diffusion matrix] there exists a positive constant $C_\sigma$ such that for all $t$ in $[0,T]$,
$$\forall x,x' \in \R^d,\ w,w'\in \R,\ |\sigma(t,x,w) - \sigma(t,x',w')|\leq C_\sigma \left(|x-x'|+|w-w'|\right).$$  
Moreover for all $(t,x)$ in $[0,T] \times \R^d$, the mapping $\sigma(t,x,\cdot) : \R \ni w \mapsto \sigma(t,x,w)$ is differentiable, $||\p_3 \sigma||_{\infty} < C_\sigma'$ and there exists a positive constant $C_\sigma''$ such that for all $t$ in $[0,T]$ and $w$ in $\R$,
$$\forall x,x' \in \R^d,\ |\p_3\sigma(t,x,w) - \p_3\sigma(t,x',w)|\leq C_\sigma'' |x-x'|^{\gamma_a'}.$$ 
Finally, the mapping $\varphi_2 : \R^d \ni x \mapsto \varphi_2(x)$ is supposed to be Lipschitz.\\

\item[(HE3) uniform ellipticity of $\sigma\sigma^*$] the function $\sigma\sigma^*$ satisfies the uniform ellipticity hypothesis:
\begin{equation*}
\exists \Lambda >1,\ \forall \zeta \in \mathbb{R}^{d}, \quad \Lambda^{-1}|\zeta|^2\leq \left[\sigma \sigma^*(t,x,w)\zeta\right] \cdot \zeta  \leq \Lambda |\zeta|^2,
\end{equation*}
for all $(t,x,w) \in  [0,T] \times \mathbb{R}^d \times \mathbb{R}$.\\
\end{description}

\textbf{Main result.} We can now state our main result:
\begin{theoreme}\label{MR}
Under assumptions \textbf{(HE)}, the system \eqref{ttarget1} admits a unique strong solution.\\
\end{theoreme}

\begin{remarque}\label{remarqueMR}
We emphasize that the result can be extended to functions $\varphi_i$ depending on time and space: the same arguments lead to the same result if the dependence of the coefficients w.r.t. the law are of the form $b(t,x,\nu) = \la\varphi_1(t,x,\cdot),\nu_t\ra$ and $\sigma(t,x,\nu) = \la\varphi_2(t,x,\cdot),\nu_t\ra$, where $(t,x)$ lies in $[0,T]\times \R^d$ and   where $(\nu_{t})_{0\leq t \leq T}$ is a family of probability measures on $\mathcal{P}_2(\R^d)$. Obviously, in order to match our assumptions, the map $\varphi_1$ is then supposed to be bounded in space and Hölder w.r.t. its third argument and the non-degeneracy assumption on the diffusion matrix has to be understood for the matrix $\left[\la\sigma(t,x,\cdot), \nu_t\ra\right] \left[\la\sigma(t,x,\cdot),\nu_t\ra\right]^*$ uniformly in $x$, $t$  in $\R^d\times \R^+$ and $\nu$ in the family of probability measures on $\R^d$.
\end{remarque}

\section{Mathematical background and strategy of proof}\label{SoP}
\subsection{The Zvonkin transformation}
For usual differential equations, it could be a very hard task to show the well posedness outside the Lipschitz framework, at least in the classical sense (see \cite{diperna_ordinary_1989} for some work in that direction). Nevertheless, when the differential system is perturbed by noise there is a phenomenon, called \emph{regularization by noise}, that allows to recover the well posedness. When the SDE is linear (in a Mckean-Vlasov sense), this has been studied first by Zvonkin \cite{zvonkin_transformation_1974} and then generalized by several authors \emph{e.g.} \cite{veretennikov_strong_1980,krylov_strong_2005,zhang_stochastic_2011} and \cite{flandoli_random_2011} for a survey. All these results rely on smoothing properties of an associated PDE and so on smoothing properties of elliptic and linear partial second order differential operator. Let us briefly explain how.\\

The strategy to recover the Lipschitz property consists in exhibiting a Zvonkin-like transformation of the equation. Let us forget for the moment the dependence of the solution of \eqref{ttarget1} w.r.t. its own law in order to illustrate the main argument. If we denote by $\mathcal{A}$ the generator of \eqref{ttarget1}, the idea is to obtain \emph{a priori} estimates on the solution of the PDE
\begin{eqnarray}\label{PDEsimple}
\p_t \bbu + \mathcal{A} \bbu = b,\text{ on }[0,T)\times \R^d,\quad \bbu_T = 0_{\R^d},
\end{eqnarray}
when $b$ and $\sigma$ are smooth functions, but depending only on regularity of $b$, $\sigma$ assumed in \textbf{(HE)}. 

This allows to consider a sequence $(\bbu^n)_{n\geq 0}$ of classical solutions of the PDE \eqref{PDEsimple} along a sequence of mollified coefficients $((\sigma\sigma^*)_n, b_n)_{n \geq 0}$. Then, by applying Itô's formula on $X_t - \bbu^n(t,X_t)$ we can remove the drift of the equation and recover an SDE whose coefficients have Lipschitz constants uniformly on the regularization procedure, so that, when letting the regularization procedure tend to infinity, the estimates pass through the limit.

When these constants can be chosen as small as $T$ is small (which follows from the boundary condition in \eqref{PDEsimple}), we then recover existence and uniqueness on small time intervals. If in addition the constants do not degenerate with the time, we can iterate the procedure and then recover existence and uniqueness on $\R^+$.\\

The smoothing properties of the PDE \eqref{PDEsimple}, are, in fact, the crucial points. It is well known that such smoothing properties are related to the noise propagation in the associated SDE through all the directions of the space. Hence, two issues arise from the non linear framework studied here: how the operator $\mathcal{A}$ looks like in our Mckean-Vlasov case, and how to regularize in the measure direction since the noise does not act in that direction. 

\subsection{PDE on space of probability measure}
Roughly speaking, we have to find a PDE that reflects the Markov structure of the underlying process. Here, the Markov property has to be understood on the space $[0,T]\times \R^d \times \mathcal{P}_2(\R^d)$, so that it seems natural to consider a PDE on this space. This sort of PDE has been recently studied independently by Buckdahn, Li, Peng and Rainer in \cite{Buckdahn_Mean-field_2014} and Crisan, Chassagneux and Delarue in \cite{chassagneux_probabilistic_2014}, it is called the Master Equation and it appears naturally when considering Mean-Field Games. What follows is essentially inspired by the second work \cite{chassagneux_probabilistic_2014}, from which we adopted some of the notations.\\

Before giving this PDE, stated on the space $[0,T]\times \R^d \times \mathcal{P}_2(\R^d)$, let us give some notions of differentiation of functions along a probability measure. The one used here has been introduced by Lions  during its lecture at the \emph{Collège de France} and can be found in Cardaliaguet's note \cite{cardaliaguet_notes_2010}. The strategy of Lions consists in lifting the function $V : \mathcal{P}_2(\R^d) \ni \nu \mapsto V(\nu) \in \R$ to a function $\mathcal{V} : \mathbb{L}_2(\Omega,\mathcal{F},\mathbb{P})\ni Z \mapsto \mathcal{V}(Z) \in \R$, $Z$ being a random variable of law $\nu$. We can then take advantage of the Hilbert structure of the $\mathbb{L}_2$ space and define, in the Frechet sense, the mapping $D\mathcal{V}$. Thanks to Riezs' representation Theorem, we can identify $D\mathcal{V}(Z)$ as $DV(\nu)(Z)$. Thus, we call the derivative of $V$ w.r.t. the law, and we denote by $\p_\nu V(\nu)$, the mapping in $\mathbb{L}_2(\R^d,\nu;\R^d)$:
$$\p_\nu V(\nu) : \R^d \ni z \mapsto \p_\nu V(\nu)(z) \in \R^d.$$

Let us emphasize that, in our case, the law interaction appears as the action of the law on some function $\varphi_i : V(\nu) = \int \varphi_i (x) \d \nu(x)$. Using the lifting argument described above we get that for any random variable $X$ and $H$ in $\mathbb{L}_2(\Omega)$:
$$V([X+\epsilon H]) = \E[\varphi(X+\epsilon H)] = \E[\varphi(X)] + \epsilon \E[\varphi'(X) \cdot H] + o(\epsilon),$$
so that 
$$\p_\nu \langle \varphi_i,\nu\rangle : \R^d \ni z \mapsto  \varphi_i'(z) \in  \R^d.$$
Finally, let us just notice that this definition justifies the choice of the space $\mathcal{P}_2(\R^d)$ for the initial data in \eqref{ttarget1}.\\

We can now state the PDE of interest:
\begin{equation}\label{targetPDE}\displaystyle
\left\lbrace\begin{array}{lll}
\displaystyle(\p_t +\mathcal{A}) \bbu(t,x,\mu) =b(t,x,\langle \varphi_1, \mu\rangle),\text{ on }[0,T]\times \R^d\times \mathcal{P}_2(\R^d).\\
\displaystyle \bbu(T,x,\mu) =0_{\R^d},
\end{array}\right.
\end{equation}
where, when setting $a:=\sigma\sigma^*$ the operator $\mathcal{A}$ is given by: for any smooth enough function $\psi : \R^+\times \R^d \times \mathcal{P}_2(\R^d)\to \R^d$

\begin{eqnarray*}
\mathcal{A}\psi(t,x,\mu) &=& \frac{1}{2}{\rm Tr}\left[a(t,x,\langle \mu,\varphi_2\rangle) \p_x^2 \psi(t,x,\mu)\right] + b(t,x,\langle \varphi_1, \mu\rangle)\p_x \psi(t,x,\mu) \\
&&+ \int b(t,x,\langle\varphi_1, \mu\rangle)\p_\mu \psi(t,x,\mu)(z) \d \mu(z) + \frac{1}{2}\int {\rm Tr}\left[a(t,x,\langle \varphi_2, \mu\rangle)\p_z(\p_\mu \psi(t,x,\mu)(z))\right] \d \mu(z).
\end{eqnarray*}
When the coefficients are smooth, it follows from \cite{chassagneux_probabilistic_2014} that such a PDE admits a classical solution $\bbu$. We refer to the aforementioned paper for more explanations on the meaning of classical solution and especially on the question about the regularity of $\p_\mu \bbu$ as an element of $\mathbb{L}_2(\R^d,\mu;\R^d)$ w.r.t. the variable $z$ and $\mu$.\\

\subsection{Smoothing properties of the PDE}
It thus remains to show the smoothing properties of the PDE. This is done under the following assumptions.

\textbf{Assumptions} $\mathbf{(H\mathcal{E})}$. We say that assumptions $\mathbf{(H\mathcal{E})}$ hold if Assumptions \textbf{(HE)} are satisfied with assumption $\mathbf{(HE2)}$ replaced by
\begin{description}
\item[(H$\mathcal{E}$2)  regularity of the diffusion matrix] there exists a positive constant $C_\sigma$ such that for all $t$ in $[0,T]$, for all $w$ in $\R$
$$\forall x,x' \in \R^d,\ |\sigma(t,x,w) - \sigma(t,x',w)|\leq C_\sigma |x-x'|^{\gamma_a},$$
for some $0<\gamma_a\leq 1$. Moreover for all $(t,x)$ in $[0,T] \times \R^d$, the mapping $\sigma(t,x,\cdot) : \R \ni w \mapsto \sigma(t,x,w)$ is differentiable, $||\p_3 \sigma||_{\infty} < C_\sigma'$ and there exists a positive constant $C_\sigma''$ such that for all $t$ in $[0,T]$ and $w$ in $\R$,
$$\forall x,x' \in \R^d,\ |\p_3\sigma(t,x,w) - \p_3\sigma(t,x',w)|\leq C_\sigma'' |x-x'|^{\gamma_a'},$$
for some $0<\gamma_a'\leq 1$. Finally, the mapping $\varphi_2 : \R^d \ni x \mapsto \varphi_2(x)\in \R$ is supposed to be $\alpha_2$-Hölder continuous, $0<\alpha_2\leq 1.$\\
\end{description}

This means that we let the diffusion coefficient $a$ be only a Hölder-continuous function w.r.t. the space and law variable. We emphasize that $\mathbf{(HE)}$ implies $\mathbf{(H\mathcal{E})}$.\\

In order to apply the Zvonkin's transformation, we do not need to solve the whole system of PDE. We can indeed regularize it and then exhibit a Lipschitz bound on the regularized solution and its space derivative depending only on known parameters in $\mathbf{(H\mathcal{E})}$.

In our context, it is possible to mollify the coefficients $b,a$ and the functions $\varphi_i$, $i=1,2,$ to obtain a sequence of smooth coefficients $(b_n,a_n)_{n\geq 1}$ (say bounded and infinitely differentiable with bounded derivatives of all order), and functions $(\varphi_1^n,\varphi_2^n)_{n\geq 1}$ (infinitely differentiable with bounded derivatives of all order greater than 1) that converges uniformly to $b,a$ and $\varphi_i$, $i=1,2$. If we denote by $\mathcal{A}^n$ the regularized version of the operator $\mathcal{A}$ and $\bbu^n$ the solution of the regularized version of \eqref{targetPDE}, we have the following result:

\begin{theoreme}\label{MRPDE}
For each $n$, the regularized system of PDEs \eqref{targetPDE} (\emph{i.e.} with $\mathcal{A}^n$ and $b_n$ instead of $\mathcal{A}$ and $b$) admits a unique classical solution $\bbu^n$, in the sense defined in \cite{chassagneux_probabilistic_2014}. Moreover, there exists a positive $\mathcal{T}_{\ref{MRPDE}}$, a positive constant $C_{\ref{MRPDE}}$ and a positive number $\delta_{\ref{MRPDE}}$ depending only on known parameters in $\mathbf{(H\mathcal{E})}$, such that, for all $(t,x,\mu)$ in $[0,T]\times \R^d \times \mathcal{P}_2(\R^d)$ and all $n$ in $\mathbb{N}^*$ the mapping $\bbu^n$ satisfies:
$$\left|\p_\mu \bbu^n(t,x,\mu) (z) \right| + \left|\p_\mu(\p_x \bbu^n(t,x,\mu) (z)) \right| + \left| \p_x \bbu^n(t,x,\mu) \right| + \left|\p^2_x \bbu^n(t,x,\mu)\right| \leq C_{\ref{MRPDE}} T^{\delta_{\ref{MRPDE}}},$$
for all $z$ in $\R^d$ and for $T$ less than $\mathcal{T}_{\ref{MRPDE}}$ 
\end{theoreme}

The smoothness of the solution in space is not new. This phenomenon is well-known and follows from the ellipticity assumption assumed on $a$.  What is more unexpected is that there are bounds obtained uniformly on the regularization procedure on the measure derivatives. Indeed, the coefficients of the PDE are not differentiable w.r.t. the argument $\mu$ and it is clear that any differentiation of $\bbu^n$ w.r.t. $\mu$ should involves the differentiation of the source term $b$ w.r.t. this argument. So, by the chain rule, the bound should contain an estimate on the derivative of $\mathcal{P}_2(\R^d) \ni \nu \mapsto \langle \varphi_i^n,\nu \rangle$, which is given by $(\varphi_i^n)'$, so that this estimate should depend on the regularization procedure. 

Nevertheless, it appears that for all $t< s$ in $[0,T]^2$ the derivative of the mapping $\mu \mapsto \langle\varphi_i, [X_s^{t,\mu}] \rangle$ can be estimated in terms of known parameters in $\mathbf{(H\mathcal{E})}$ (in fact, combining an additional estimate on $\p_z \p_\mu \bbu$ together with Arzelà-Ascoli Theorem, we are able to show that the estimate on $\p_\mu \bbu$ holds for the mild solution of \eqref{targetPDE}).

We hence have a smoothing property in the measure space without any action of the Laplacian in that direction. This follows from the fact that the function $\varphi_i$ is integrated against the law of the process so that there still is a Gaussian convolution of the initial data ``$\mu$'' at any time $s>t$. Therefore, we recover the spatial smoothing.

To the best of our knowledge, this result is new, especially since we do not add any noise on the space of measure. This last aspect has been studied in \cite{carmona_mean_2014} where Mean Field Games with common noise are investigated. Roughly speaking, the Authors showed that common noise on the original system of interacting players translates into McKean-Vlasov system with random law (the family of probability measure of the underlying stochastic process having now a stochastic dynamic) which allows them to recover existence and  uniqueness of Nash equilibrium.

Also, in the same spirit as us, David R. Ba\~{n}os studies in \cite{banos_bismut-elworthy-li_2015} the Malliavin differentiability of processes having the same dynamic as \eqref{ttarget1} with Lipschitz coefficients. Although he does not consider explicitly a regularization phenomenon in the measure direction (the functions $\varphi_i$ are continuously differentiables with bounded Lipschitz derivatives), he shows that the space regularization phenomenon still holds so that the mapping $x \mapsto \langle \phi, [X_s^{t,\delta_x}]\rangle \in \R$, $\phi$ in $\mathbb{L}_2([X_s^{t,\delta_x}])$ is weakly differentiable for any $s>t$ thanks to a stochastic perturbation approach of Bismut type.\\

Finally, let us emphasize that Remark \ref{remarqueMR} also applies for Theorem \ref{MRPDE}. It seems that the main structural assumption that allows to recover the spatial smoothing property in the measure direction comes from the particular ``polynomial '' dependence of the coefficients w.r.t. the measure. 


\section{Proof of the main result}\label{PMR}

We prove our main result by using a Picard's iteration. Let $m$ be a positive integer, set $(X_t^{\mu})^0 = X_0$ for all $t$ in $[0,T]$ and define $(X_t^\mu)^{m+1}$ as the solution of:
$$(X_t^\mu)^{m+1} = X_0 + \int_0^t b(r,(X_r^{\mu})^m,\langle \varphi_1, [(X_r^{\mu})^m]\rangle ) \d r + \int_0^t\sigma(r,(X_r^{\mu})^m,\langle \varphi_2, [(X_r^{\mu})^m]\rangle) \d B_r,\quad X_0 \sim \mu.$$
In order to remove the singular drift, we now have to use an Itô's Formula that matches our framework, \emph{i.e.} stated on $[0,T]\times \R^d \times \mathcal{P}_2(\R^d)$. This formula, involving chain rule for functions defined on $\mathcal{P}_2(\R^d)$, can be found in Section 3 of \cite{chassagneux_probabilistic_2014}. By applying it to 
$$(X_t^{\mu})^{m+1} - \bbu^n(t,(X_t^{\mu})^{m+1},[(X_t^{\mu})^{m+1}]),$$
we obtain that
\begin{eqnarray}\label{VirtualSDE}
(X_t^{\mu})^{m+1} &=& X_0 - \bbu^n(0,X_0,\mu) - \bbu^n(t,(X_t^{\mu})^{m+1},\left[(X_t^{\mu})^{m+1}\right]) \\
&&+ \int_0^t  \sigma(s,(X_s^{\mu})^m,\langle \varphi_2,[(X_s^{\mu})^{m}] \rangle) \left[1-\p_2\bbu^n(s,(X_s^{\mu})^{m+1},\left[(X_s^{\mu})^{m+1}\right])\right] \d B_s + \mathcal{R}_t^m(n),\notag
\end{eqnarray}
where
\begin{eqnarray*}
\mathcal{R}_t^m(n) &=&  \int_0^t b_{n}(s,(X_s^{\mu})^{m},\langle \varphi_1^n,[(X_s^{\mu})^{m}] \rangle)-b(s,(X_s^{\mu})^{m},\langle \varphi_1,[(X_s^{\mu})^{m}] \rangle)\\
&& + (\mathcal{A}^{n}-\mathcal{A})\bbu^n(s,(X_s^{\mu})^{m+1},[(X_s^{\mu})^{m+1}])) \d s.
\end{eqnarray*}
Hence,
\begin{eqnarray*}
|(X_t^{\mu})^{m+1} -(X_t^{\mu})^{m}| &\leq & |\bbu^n(t,(X_t^{\mu})^{m+1},\left[(X_t^{\mu})^{m+1}\right]) -\bbu^n(t,(X_t^{\mu})^{m},\left[(X_t^{\mu})^{m}\right]) | \\
&& + \Bigg|\int_0^t\bigg\{  \sigma(s,(X_r^{\mu})^m,\langle \varphi_2,[(X_s^{\mu})^{m}] \rangle)\left[ 1 - \p_2\bbu^n(s,(X_s^{\mu})^{m+1},\left[(X_s^{\mu})^{m+1}\right])\right]\\
&&\quad  -\sigma(s,(X_s^{\mu})^{m-1},\langle \varphi_2,[(X_s^{\mu})^{m}] \rangle) \left[1-\p_2\bbu^n(s,(X_s^{\mu})^{m},\left[(X_s^{\mu})^m\right])\right]\bigg\}\d B_s\Bigg| \\&& + |\mathcal{R}_t^{m}(n)|+|\mathcal{R}_t^{m-1}(n)|.\notag
\end{eqnarray*}

Let us now emphasize that when $T<\mathcal{T}_{\ref{MRPDE}}$, Theorem \ref{MRPDE} implies that for all $(t,x)$ in $[0,T]\times \R^d$, all $Z$, $Z'$ in $\mathbb{L}_2(\Omega)$:
\begin{eqnarray*}
&&\left|\p_2^l\bbu^n(t,x,[Z]) - \p_2^l\bbu^n(t,x,[Z'])\right|\\
&& \quad  \leq\left| \int_0^1 \E\big[\p_3 \p_2^l \bbu^n(t,x,[(1-\lambda)Z+ \lambda Z']) (Z-Z') \big] \d \lambda\right|\\
&&\quad \leq C_{\ref{MRPDE}} T^{\delta_{\ref{MRPDE}}} \E[|Z-Z'|].
\end{eqnarray*}
for $l=0,1$, so that for any $T$ less  than $\mathcal{T}_{\ref{MRPDE}}$, there exists a positive $\delta$, depending on known parameters in \textbf{(HE)}, such that:

\begin{eqnarray*}
\E\sup_{t\leq T}|(X_t^{\mu})^{m+1} -(X_t^{\mu})^{m}|^2 &\leq & C T^{\delta} \E\sup_{t\leq T} |(X_t^{\mu})^{m+1} -(X_t^{\mu})^{m} |^2 \\
&& + C\int_0^t \E \sup_{r\leq s}| (X_r^{\mu})^m-(X_r^{\mu})^{m-1}|^2 \d s  + 2\E |\mathcal{R}_T^{m}(n)|^2+2\E |\mathcal{R}_T^{m-1}(n)|^2,\notag
\end{eqnarray*}
Since for any $m$, $\E|\mathcal{R}_T^m(n)|^2$ tends uniformly to $0$ as $n$ tend to infinity, we can let $n$ tend to infinity in the right hand side of the equation above and we get that 

\begin{eqnarray*}
(1-C T^{\delta})\E\sup_{t\leq T} |(X_t^{\mu})^{m+1} -(X_t^{\mu})^{m}|^2 &\leq &  C_T\int_0^t \E \sup_{r\leq s} | (X_r^{\mu})^m-(X_r^{\mu})^{m-1}|^2 \d s.\notag
\end{eqnarray*}
Finally, we can find a positive $\mathcal{T}$ depending on known parameters in $\mathbf{(HE)}$ such that for any $T$ less than $\mathcal{T}$:
\begin{eqnarray*}
\E \sup_{t\leq T}|(X_t^{\mu})^{m+1} -(X_t^{\mu})^{m}|^2 &\leq & C_T \int_0^T \E \sup_{r\leq s} | (X_r^{\mu})^m-(X_r^{\mu})^{m-1}|^2 \d s.\notag
\end{eqnarray*}
By induction, we deduce that for any $T$ less than $\mathcal{T}$:
\begin{eqnarray*}
\E \sup_{t\leq T}|(X_t^{\mu})^{m+1} -(X_t^{\mu})^{m}|^2 &\leq & \frac{C_T^{m+1} T^{m}}{m!} \E \sup_{t\leq T} | (X_t^{\mu})^1-(X_t^{\mu})^{0}|^2 \leq C' \frac{C_T^{m+1} T^{m}}{m!}.\notag
\end{eqnarray*}
So that $(X_t^\mu)^m$ converges almost surely to a solution $X_t^\mu$ of \eqref{ttarget1}. We deduce the uniqueness part from the previous computations. We hence have existence and uniqueness of a solution on $[0,\mathcal{T}]$. We can then iterate the construction and obtain the result for all $T$ in $\R^+$.\\

\section{Estimation on the solution of the PDE: proof of Theorem \ref{MRPDE}}\label{EotS}

\textbf{Notations.} From now, we let $C, C', c, c',\tilde{C}, \tilde{C}',\ldots$ be some positive constants  depending only on known parameters in \textbf{(H$\mathcal{E}$)}.\\

Let us first reduce the problem. We emphasize that any component of the $d$-dimensional solution of the system of PDEs \eqref{targetPDE} above can be described by the solution of:
\begin{equation}\label{targetPDE2}\displaystyle
\left\lbrace\begin{array}{lll}
\displaystyle(\p_t + \mathcal{A})u(t,x,\mu)  =\tilde{b}(t,x,\langle \varphi_1,\mu\rangle),\text{ on }[0,T]\times \R^d\times \mathcal{P}_2(\R^d)\\
\displaystyle u(T,x,\mu) =0,
\end{array}\right.
\end{equation}
where $\tilde{b} : \R^+\times \R^d \times \R \to \R$ plays the role of one of the components of $b$. Hence, we only have to prove the estimates in Theorem \ref{MRPDE} for the function $u$ defined above. Next, we have that Theorem \ref{MRPDE} is stated under regularized framework. For the sake of clarity, we forget the superscript $n$ that follows from the regularization procedure in the following and we suppose that the following assumptions hold.

\textbf{Assumptions} $\mathbf{(H\mathcal{E}R)}$.  We say that assumptions $\mathbf{(H\mathcal{E}R)}$ hold if assumptions $\mathbf{(H\mathcal{E})}$ hold true and $b, \sigma, \varphi_1, \varphi_2$ are infinitely differentiable functions with bounded derivatives of all order, greater than 1 for the functions $\varphi_1$ and $\varphi_2$.\\

We know from \cite{chassagneux_probabilistic_2014}  that under \textbf{(H$\mathbf{\mathcal{E}}$R)} this PDE admits a unique classical solution. Let us now give a suitable representation of this solution.\\

Under assumptions $\mathbf{(H\mathcal{E}R)}$ it follows from the Sznitman's note \cite{sznitman_topics_1991} that equation \eqref{ttarget1} admits a unique strong solution. For any $(t,\mu)$ in $[0,T] \times \mathcal{P}_2(\R^d)$ and $x$ in $\R^d$, its flow is the solution on $[t,T]$ of 
\begin{equation}\label{ttarget2}
X_s^{t,x,\mu} = x + \int_t^s b(r,X_r^{t,x,\mu},\langle \varphi_1, [X^{t,\mu}_r]\rangle) \d r + \int_t^s\sigma(r,X_r^{t,x,\mu},\langle \varphi_2, [X^{t,\mu}_r]\rangle) \d B_r.
\end{equation}
Given the family of marginals $[X^{t,\mu}]: = ([X_s^{t,\mu}])_{t\leq s \leq T}$ of the solution of \eqref{ttarget1}, we can consider the stochastic system \eqref{ttarget2} as a linear system parametrized by the time dependent parameter $[X^{t,\mu}]$. We can then define for all $(s',y')$ in $[0,T] \times \R^d$ the process $X^{s',y',[X^{t,\mu}]}$ as the solution of \eqref{ttarget2} on $[s',T]$ with starting point $y'$ at time $s'$ and whose coefficients depend on $[X^{t,\mu}]$ and we denote by $\mL^{t,\mu}$ its generator.\\

It is then clear, thanks to the well posedness of \eqref{ttarget2} under $\mathbf{(H\mathcal{E}R)}$, that  $X^{t,x,[X^{t,\mu}]}=X^{t,x,\mu}$. Finally, from classical theory of linear SDEs, the flow $X^{t,x,[X^{t,\mu}]}$ admits a transition density $p$ which is also parametrized by $[X^{t,\mu}]$.

Since from the arguments of \cite{chassagneux_probabilistic_2014} we have that, for all $(t,x,\mu) \in [0,T]\times \R^d\times \mathcal{P}_2(\R)$, the solution of the PDE \eqref{targetPDE2} writes:
\begin{equation*}
u(t,x,\mu) = \E \int_t^T \tilde{b}(s,X_s^{t,x,\mu},\langle \varphi_1 , [X_s^{t,\mu}]\rangle) \d s,
\end{equation*}
we deduce from the previous discussion that
\begin{equation}\label{eq:repsol}
u(t,x,\mu) = \int_t^T \int_{\R^d} \tilde{b}(s,y,\langle \varphi_1, [X_s^{t,\mu}]\rangle) p([X^{t,\mu}];t,x;s,y) \d y \d s.
\end{equation}

In order to keep the notations clear, we only mention the dependence of $p$ w.r.t. the initial data $(t,\mu)$ of \eqref{ttarget1} in the following and we forget its first argument when the starting time of \eqref{ttarget1} and \eqref{ttarget2} are the same. Hence, for all $(s',y')$ and $(s,y)$ in $[t,T]\times \R^d$: 
$p(t,\mu;s',y';s,y) : = p([X^{t,\mu}];s',y;s,y)$ and  $p(\mu;t,y';s,y) := p(t,\mu;t,y';s,y)$.

When differentiating the function $u$ in the measure direction, we have to differentiate the integrand in the expression \eqref{eq:repsol} in that direction. Thus, we have to estimate a quantity of the form 
$$\p_\mu \langle \varphi , [X_s^{t,\mu}] \rangle,\quad s\in (t,T].$$

Then, in the following, for any $\alpha$-Hölder function $\phi$, $\alpha \in (0,1]$, we denote by $v$ the mapping:
\begin{equation}\label{def:v}
v: (t<s,\mu,\phi) \in [0,T]^2\times  \mathcal{P}_2(\R^{d}) \mapsto v_s(t,\mu,\phi) = \langle \phi , [X_s^{t,\mu}]\rangle,
\end{equation}
and we prove in Section \ref{Eov} the following Proposition.

\begin{Proposition}\label{prop:estiv}
Suppose that assumptions $\mathbf{(H\mathcal{E}R)}$ hold, let $\phi$ be some $\alpha$-Hölder function from $\R^d$ to $\R$, let $t$ in $[0,T]$ and let us denote by $\mu$ the law of the solution of \eqref{ttarget1} at time $t$. There exist a positive number $\mathcal{T}_{\ref{prop:estiv}}$ and a positive constant $C_{\ref{prop:estiv}}$, depending only on known parameters in $\mathbf{(H\mathcal{E})}$, such that for all $z$ in $\R^d$ and $s$ in $(t,T]$:
\begin{eqnarray}\label{esti:derivdev}
|\p_\mu v_s(t,\mu,\phi)(z)| \leq C (s-t)^{(-1+\alpha)/2},
\end{eqnarray}
for all $T$ less than $\mathcal{T}_{\ref{prop:estiv}}$.
\end{Proposition}

Moreover, we have from \eqref{eq:repsol} that the derivative of $u$ or $\p_x u$ along the measure involves the derivative of the transition density $p$ or $\p_x p$ along the measure. We then have to obtain suitable control of these quantities. Here, these controls are summarized by the following Proposition whose proof is postponed to Section \ref{Parametrix}.

\begin{Proposition}\label{prop:estidensity}
Suppose that assumptions $\mathbf{(H\mathcal{E}R)}$ hold, let $t$ in $[0,T]$ and let us denote by $\mu$ the law of the solution of \eqref{ttarget1} at time $t$. Then, for all $x$ in $\R^d$, for all $(s,y)$ in $(t,T]\times \R^d$ and all $z$ in $\R^d$: $p(\mu;t,x;s,y) \leq \hat{p}_c(t,x;s,y)$ where $\hp_c$ is the Gaussian like kernel defined by:
\begin{equation}\label{eq:defGausskern}
\hp_c(t,x;s,y) = \frac{c}{(s-t)^{d/2}} \exp\left(-c\frac{|y-x|^2}{(s-t)}\right),
\end{equation}
where $c$ depends on known parameters in $\mathbf{(H\mathcal{E})}$ only. Moreover, there exist two positive constants $C_{\ref{prop:estidensity}}$ and $C_{\ref{prop:estidensity}}'$, depending only on known parameters in $\mathbf{(H\mathcal{E})}$, such that for all $x$ in $\R^d$, for all $(s,y)$ in $(t,T]\times \R^d$ and all $z$ in $\R^d$
\begin{eqnarray}
&&\p_\mu p(\mu;t,x;s,y)(z) \leq C_{\ref{prop:estidensity}} \sum_{i=1}^2(s-t)^{(\alpha_i-1)/2} \overline{\p_\mu v}^i_{(\alpha_i-1)/2}(\mu;t,s)(z) \hat{p}_c(t,x;s,y),\label{prop:estidensity2}\\
&& \p_\mu \p_x p(\mu;t,x;s,y)(z) \leq C_{\ref{prop:estidensity}}' \sum_{i=1}^2 (s-t)^{\alpha_i/2-1} \overline{\p_\mu v}^i_{\alpha_i/2-1}(\mu;t,s)(z) \hat{p}_c(t,x;s,y),\label{prop:estidensity3}
\end{eqnarray}
where we used the abusive notation 
\begin{eqnarray*}
&&\overline{\p_\mu v}_\gamma^i(\mu;t,s)(z) =  \sup_{r\in[t,s]}\big\{(r-t)^{-\gamma}|\p_\mu v_r(t,\mu,\varphi_i)(z) |\big\}.
\end{eqnarray*}
\end{Proposition}

We have now all the ingredients to complete the proof. Thanks to estimates \eqref{esti:derivdev} on $\p_\mu v$ and \eqref{prop:estidensity2} on $\p_\mu p$ we deduce that we can invert the differentiation and integration operators when differentiating the right hand side of \eqref{eq:repsol} w.r.t. the measure. Hence, the derivative of $u$ in the measure direction writes, at any point $z$ of $\R^d$:
\begin{eqnarray}
\p_\mu u(t,x,\mu)(z) &=& \int_t^T \int_{\R^d} \p_3 \tilde{b}(s,y,\langle \varphi_1, [X_s^{t,\mu}]\rangle) \p_\mu \langle \varphi_1, [X_s^{t,\mu}]\rangle(z)  p(\mu;t,x;s,y) \d y \d s\notag\\
&+& \int_t^T \int_{\R^d} \tilde{b}(s,y,\langle \varphi_1, [X_s^{t,\mu}]\rangle) \p_\mu p(\mu;t,x;s,y)(z) \d y \d s,
\end{eqnarray}
and satisfies, thanks to estimates \eqref{prop:estidensity2} and \eqref{esti:derivdev}

\begin{eqnarray*}
&&\left|\p_\mu u(t,x,\mu)(z)\right| \\
&&\leq  C \int_t^T \int_{\R^d} ||\p_3\tilde{b}||_{\infty} (s-t)^{(-1+\alpha)/2}  p(\mu;t,x;s,y) \d y \d s\notag\\
&& \quad +  \int_t^T \int_{\R^d}  ||\tilde{b}||_{\infty} C_{\ref{prop:estidensity}} \sum_{i=1}^2(s-t)^{(-1+\alpha_i)/2} \overline{\p_\mu v}^i_{(\alpha_i-1)/2}(\mu;t,s)(z) \hat{p}_c(t,x;s,y) \d y \d s.
\end{eqnarray*}

Therefore we can deduce from \eqref{esti:derivdev} the there exist a positive constant $C'$ and a positive number $\delta$,  depending only on known parameters in $\mathbf{(H\mathcal{E})}$, such that:

\begin{eqnarray*}
\left|\p_\mu u(t,x,\mu)(z)\right| &\leq & C' T^{\delta}.
\end{eqnarray*}

Now, we have: 
\begin{equation}
\p_x u(t,x,\mu) = \int_t^T \int_{\R^d} \tilde{b} (s,y,\langle \varphi_1, [X_s^{t,\mu}]\rangle) \p_x p(\mu;t,x;s,y) \d y \d s.
\end{equation}

Hence, we can differentiate the mapping $\p_xu$ along the measure and by using the same arguments as above, with estimate \eqref{prop:estidensity3} on $\p_\mu \p_x p$ instead of \eqref{prop:estidensity2}, we obtain that 
\begin{eqnarray*}
\left|\p_\mu (\p_x u(t,x,\mu))(z)\right| &\leq & C'' T^{\delta'},
\end{eqnarray*}
for some positive constant $C''$ and positive number $\delta'$,  depending only on known parameters in $\mathbf{(H\mathcal{E})}$.\\

Finally, the estimates on $\p_x u $ and $\p^2_{x}u$ can be obtained by classical arguments, when viewing the argument $\mu$ as a parameter. See \emph{e.g.} \cite{friedman_partial_1964}. This concludes the proof of Theorem \ref{MRPDE}.

\section{Differentiation and estimation of $v$: proof of Proposition \ref{prop:estiv}}\label{Eov}
With the notations defined in the previous section, we have that  for all $s$ in $(t,T]$,
\begin{eqnarray*}
v_s(t,\mu,\phi) &=& \int_{\R^d} \phi(y) \int_{\R^d} p(\mu; t,x;s,y) \d \mu(x) \d y\\
&:=& \int_{\R^d} \phi(y)P(t,\mu,\mu;s,y) \d y,
\end{eqnarray*}
where the function $P$ is the function
$$P : (t<s,\lambda,\nu,y) \in [0,T]^2\times \mathcal{P}_2(\R^d) \times \mathcal{P}_2(\R^d) \times \R^d  \mapsto P (t,(\lambda,\nu);s,y) =   \int_{\R^d} p(\lambda;t,x;s,y) \d \nu(x).$$

With this notation and by using the fact that $p$ is a density we have:

\begin{eqnarray*}
\p_\mu v_s(t,\mu,\phi) &=& \p_\mu  \int_{\R^d}  \phi(y) P(t,\mu,\mu; s,x) \d y\\
&=&  \left[\p_\lambda\int_{\R^d} \phi(y)  P(t,\lambda,\mu;s,y)  \d y \right]_{\lambda = \mu}+ \left[\p_\nu\int_{\R^d}  \phi(y)  P(t,\mu,\nu;s,y) \d y\right]_{\nu=\mu}\\
&=& \left[\p_\lambda\int_{\R^d} \phi(y)  P(t,\lambda,\mu;s,y)  \d y\right]_{\lambda = \mu} +\int_{\R^d}  (\phi(y)-\phi(\xi))   [\p_\nu  P(t,\mu,\nu;s,y)]_{\nu = \mu} \d y,
\end{eqnarray*}
whatever $\xi$ in  $\R^d$. Since by Fubini's Theorem we have

\begin{eqnarray*}
\p_\lambda\int_{\R^d} \phi(y)  P(t,\lambda,\mu;s,y)  \d y &=&   \p_{\lambda} \int_{\R^d} \int_{\R^d} (\phi(y) -\phi(x)) p(\lambda;t,x;s,y)\d \mu(x) \d y\\
&& +\p_{\lambda} \int_{\R^d}  \int_{\R^d}  \phi(x)p(\lambda;t,x;s,y)\d \mu(x) \d y\\
&=&   \p_{\lambda} \int_{\R^d} \int_{\R^d} (\phi(y) -\phi(x)) p(\lambda;t,x;s,y)\d \mu(x) \d y\\
&& +\p_{\lambda} \int_{\R^d}  \phi(x) \int_{\R^d}  p(\lambda;t,x;s,y) \d y \d \mu(x)\\
&=&   \p_{\lambda} \int_{\R^d} \int_{\R^d} (\phi(y) -\phi(x)) p(\lambda;t,x;s,y)\d \mu(x) \d y,
\end{eqnarray*}
and since by definition

\begin{eqnarray*}
\p_\nu  P(t,\mu,\nu;s,y)(\cdot) &=& \p_x p(\mu;t,\cdot;s,y),
\end{eqnarray*}
we deduce that for all $\xi$ in $\R^d$, the derivative $\p_\mu v$ taking at any point $z$ in $\R^d$ writes

\begin{eqnarray}\label{proofMR:ecriturev}
\p_\mu v_s(t,\mu,\phi)(z) &=&\int_{\R^d} \int_{\R^d} (\phi(y) -\phi(x)) \p_\mu p(\mu;t,x;s,y)(z) \d \mu(x) \d y \notag\\
&& +\int_{\R^d}  (\phi(y)-\phi(\xi)) \p_x p(\mu;t,z;s,y) \d y.
\end{eqnarray}

So that for any given $z$ in $\R^d$, by choosing $\xi=z$ we get 
\begin{eqnarray*}
\p_\mu v_s(t,\mu,\phi)(z) &=&\int_{\R^d} \int_{\R^d} (\phi(y) -\phi(x)) \p_\mu p(\mu;t,x;s,y)(z) \d \mu(x) \d y \\
&& +\int_{\R^d}  (\phi(y)-\phi(z)) \p_x p(\mu;t,z;s,y) \d y.
\end{eqnarray*}

Thanks to the estimates on the transition density $p$ and its derivatives in the measure direction from Proposition \ref{prop:estidensity}, Fubini's Theorem, regularity assumed on $\phi$ and by using the Gaussian decay of $\hat{p}_c$ \footnote{\emph{i.e.} the inequality: $\forall \eta>0,\ \forall q>0,\ \exists \bar{C}>0 \text{ s.t. } \forall \sigma >0,\ \sigma^{q} e^{-\eta\sigma}\leq \bar{C} $.}  we obtain the following bound:
\begin{eqnarray}
\left| \p_\mu v_s(t,\mu,\phi)(z) \right| &\leq & C\Bigg\{\sum_{i=1}^2 (s-t)^{(\alpha_i-1)/2} \overline{\p_\mu v}^i_{{(\alpha_i-1)/2}}(\mu;t,s)(z) \int_{\R^d} \int_{\R^d} |y -x|^{\alpha} \hp_c(t,x;s,y) \d \mu(x) \d y\notag \\
&& +\int_{\R^d}  |y -z|^{\alpha} (s-t)^{-1/2} \hp_c(t,z;s,y) \d y\notag\Bigg\}\\
&\leq & C'\Bigg\{\sum_{i=1}^2 (s-t)^{(-1+\alpha_i+\alpha)/2} \overline{\p_\mu v}^i_{{(\alpha_i-1)/2}}(\mu;t,s)(z) +(s-t)^{(-1+\alpha)/2}\Bigg\}\label{proofMR:esti1},
\end{eqnarray}
which holds true for any $\alpha$-Hölder function $\phi$. Then, by choosing $\phi=\varphi_1$ (and so $\alpha=\alpha_1$), by multiplying both sides by $(s-t)^{(1-\alpha_1)/2}$, we deduce from a circular argument that there exists a positive time $\mathcal{T}'$ depending only on known parameters in $\mathbf{(H\mathcal{E})}$, such that for all $T$ less than $\mathcal{T}'$:
\begin{eqnarray}\label{proofMR:esti2}
(s-t)^{(1-\alpha_1)/2}\left| \p_\mu v_s(t,\mu,\varphi_1)(z) \right| &\leq & C''\bigg\{  (s-t)^{\alpha_2/2} \overline{\p_\mu v}^2_{{(-1+\alpha_2)/2}}(\mu;t,s)(z) + 1\bigg\}.
\end{eqnarray}
By plugging this estimate in \eqref{proofMR:esti1} and by iterating this argument (choosing $\phi=\varphi_2$ so that $\alpha=\alpha_2$, then multiplying both sides by $t^{(1-\alpha_2/2)}$ and using a circular argument)  we obtain that there exists a positive time $\mathcal{T}''$, depending only on known parameters in $\mathbf{(H\mathcal{E})}$, such that for all $T$ less than $\mathcal{T}''$:
\begin{eqnarray}\label{proofMR:esti3}
(s-t)^{(1-\alpha_2)/2}\left| \p_\mu v_s(t,\mu,\varphi_2)(z) \right| &\leq & C''.
\end{eqnarray}
Again, by plugging this estimate in \eqref{proofMR:esti2} and then using the resulting estimate together with \eqref{proofMR:esti3} in \eqref{proofMR:esti1}, we finally deduce that there exists a positive time $\mathcal{T}$, depending only on known parameters in $\mathbf{(H\mathcal{E})}$, such that for all $T$ less than $\mathcal{T}$:
\begin{eqnarray*}
(s-t)^{(1-\alpha)/2}\left| \p_\mu v_s(t,\mu,\phi)(z) \right| &\leq & C''',
\end{eqnarray*}
which concludes the proof.

\section{Estimation of the transition density $p$}\label{Parametrix}
This section is dedicated to the proof of Proposition \ref{prop:estidensity} and so, to the study of the transition density $p$ of the flow \eqref{ttarget2}. Under $\mathbf{(H\mathcal{E}R)}$ it is clear that for all initial data $(t,\mu)$ in $[0,T]\times \mathcal{P}_2(\R^d)$, equation \eqref{ttarget1} admits a unique solution $X^{t,\mu}$. Thus, we can suppose that the family of probability measures $[X^{t,\mu}]$ acts as a time dependent parameter in \eqref{ttarget2}, so that the unique solution of \eqref{ttarget2} has a classical transition density $p$ parametrized by the family of probability measures $[X^{t,\mu}]$. Once the law dependence is fixed, we can now use a classical parametrix expansion of McKean-Singer type \cite{mckean_jr._curvature_1967} for linear processes in order to represent the transition density $p$.

The parametrix expansion of Mckean and Singer is based on the following observation: in small time, the transition density of a (smooth enough) process should be closed enough to the transition density of the associated frozen process (\emph{i.e.} whose coefficients are constants and fixed at the (final) value of the process).  Hence, the transition density of interest can be expanded in terms of the frozen transition density. Since the frozen transition density usually enjoys well known properties, \emph{e.g.} it has an explicit form or can be estimated by explicit (and nice) functions, this expansion allows to estimate the original transition density. As a consequence, this method requires a good knowledge of the frozen transition density and on the associated frozen process.

Thus, this section is organized as follows: we first introduce in subsection \ref{subsec:frozproc} the frozen process and its associated transition density and give its explicit expression. Then, we give the estimates on the frozen transition density and its derivative. In subsection \ref{subsec:paramexp} we show how the transition density of \eqref{ttarget2} can be expanded in terms of the transition density of the frozen process. Hence, we obtain an explicit expression of the transition density $p$ of \eqref{ttarget2} which can be estimated. These estimations are done in subsection \ref{subsec:estip} and lead to proof of Proposition \ref{prop:estidensity}. Finally, we suppose throughout this section that $T<1$.

\subsection{The frozen system}\label{subsec:frozproc}
Let $(t,\mu)$ in $[0,T]\times \mathcal{P}_2(\R^d)$, for any point $\xi$ in $\R^d$, we define the frozen flow as the solution of: 
\begin{equation}\label{froz}
\tX_s^{s',y',\xi,[X^{t,\mu}]} = y' + \int_{s'}^s b(r,\xi,\langle\varphi_1, [X_r^{t,\mu}]\rangle) \d r + \int_{s'}^s\sigma(r,\xi,\langle\varphi_2, [X_r^{t,\mu}]\rangle) \d B_r,
\end{equation}
for all $(s',y')$ in $[t,T]\times \R^d$. Under $\mathbf{(H\mathcal{E}R)}$, it is clear that this flow exists and is unique, moreover, it has a transition density $\tp$ defined for all $(s,y)$ in $[t,T]\times \R^d$ by:
\begin{equation}\label{eq:expdetildep}
\tp^{\xi}(t,\mu;s',y';s,y) = \frac{1}{(2\pi)^{d/2}} \big[\det[a_{s',s}^{\xi}(t,\mu)]\big]^{-1/2} \exp\l(-\frac{1}{2}\bigg|\big[a_{s',s}^{\xi}(t,\mu)\big]^{-1/2}\big(y-y'-m_{s',s}^{\xi}(t,\mu)\big)^*\bigg|^2 \r),
\end{equation}
where we adopted the same convention of notations for $\tp$ as for $p$ and where
\begin{eqnarray}
m_{s',s}^{\xi}(t,\mu) = \int_{s'}^s b(r,\xi,\langle \varphi_1, [X_r^{t,\mu}] \rangle) \d r,\\
a_{s',s}^{\xi}(t,\mu) = \int_{s'}^s a(r,\xi,\langle \varphi_2, [X_r^{t,\mu}] \rangle) \d r.
\end{eqnarray}
The frozen transition density \eqref{eq:expdetildep} admits Gaussian type bounds, namely, we prove in Appendix \ref{sec:prooflemmas} the following result. 
\begin{Proposition}\label{prop:estifroedensity1} Suppose that hypothethis $\mathbf{(H\mathcal{E}R)}$ holds. Let $t$ in $[0,T]$ and let $\mu$ denotes the law of the process \eqref{ttarget1} at time $t$. Then:

\begin{enumerate}[$\bullet$]
\item there exists a positive constant $\bar{C}_{\ref{prop:estifroedensity1}}$, depending only on known parameters in $\mathbf{(H\mathcal{E})}$, such that
\begin{equation}\label{gaussboundptilde} 
\forall s<s'\in [0,T]^2,\ \forall y,y'\in \R^d,\ \forall \xi \in \R^{d},\quad  \tp^\xi(t,\mu;s',y';s,y) \leq \bar{C}_{\ref{prop:estifroedensity1}}\hp_{c}(s',y';s,y),
\end{equation}
where $\hp_c$ is the Gaussian like kernel defined by \eqref{eq:defGausskern};
\item there exist two positive constants $C_{\ref{prop:estifroedensity1}}$ and $C_{\ref{prop:estifroedensity1}}'$, depending on known parameters in $\mathbf{(H\mathcal{E})}$ only, such that for all $s$ in $(t,T]$, for all $\xi$ in $\R^d$ and all $x,y$ in $\R^d$:
\begin{eqnarray}
&&\big|\p_\mu \tp^{\xi}(\mu;t,x;s,y)(z)\big| \leq C_{\ref{prop:estifroedensity1}}\sum_{i=1}^2  (s-t)^{(\alpha_i-1)/2} \overline{\p_\mu v}_{(\alpha_i-1)/2}^i(\mu;t,s')(z)\hat{p}_c(t,x;s,y) ,\label{estipderivmu}\\
&&\big|\p_\mu \p_x \tp^{\xi}(\mu;t,x;s,y)(z)\big| \leq C_{\ref{prop:estifroedensity1}}' \sum_{i=1}^2  (s-t)^{\alpha_i/2-1} \overline{\p_\mu v}_{(\alpha_i-1)/2}^i(\mu;t,s)(z)\hat{p}_c(t,x;s,y), \label{estipderivxmu}
\end{eqnarray}
for all $z$ in $\R^d$;
\item for all $s'<s$ in $(t,T]^2$, for all $y',y$ in $\R^d$:
\begin{eqnarray*}
&&\big|\p_\mu \tp^{\xi}(t,\mu;s',y';s,y)(z)\big| \leq C \sum_{i=1}^2  (s'-t)^{(\alpha_i-1)/2}\overline{\p_\mu v}_{(\alpha_i-1)/2}^i(\mu;t,s) (z) \hat{p}_c(s',y';s,y),\\
&&\big|\p_\mu \p_{y'}\tp^{\xi}(t,\mu;s',y';s,y)(z)\big| \leq C'(s-s')^{-1/2} \sum_{i=1}^2  (s'-t)^{(\alpha_i-1)/2} \overline{\p_\mu v}_{(\alpha_i-1)/2}^i(\mu;t,s')(z)  \hat{p}_c(s',y';s,y),\\
&&\big|\p_\mu \p^2_{y'}\tp^{\xi}(t,\mu;s',y';s,y)(z)\big| \leq C''(s-s')^{-1} \sum_{i=1}^2  (s'-t)^{(\alpha_i-1)/2} \overline{\p_\mu v}_{(\alpha_i-1)/2}^i(\mu;t,s')(z)  \hat{p}_c(s',y';s,y),
\end{eqnarray*}
for all $z$ in $\R^d$.
\end{enumerate}
\end{Proposition}

Finally, we have that the generator of the frozen flow \eqref{froz} is given by
\begin{equation}
\tilde{\mL}^{\xi,t,\mu}_{s',y'} : = b(s',\xi,\langle \varphi_1, [X_{s'}^{t,\mu}] \rangle) \p_{y'} + \frac{1}{2}{\rm Tr}\left[a(s',\xi,\langle \varphi_2, [X_{s'}^{t,\mu}] \rangle) \p_{y'}^2\right].
\end{equation}
Above, the subscript $(s',y')$ means that the coefficients of the operator are evaluated at time $s'$ and that the differentiation operator acts on the space variable $y'$.

\subsection{The parametrix expansion}\label{subsec:paramexp}
We now give the parametrix representation of the transition density $p$ of \eqref{ttarget2}. Proof of such a result is classical, we nevertheless wrote it since it allows to understand the crucial estimates, which will be useful in the sequel.

\begin{Proposition}\label{prop:parametrixrep}
There exists a smoothing kernel $H : [0,T]\times \mP_2(\R^d)\times [0,T] \times \R^d \times [0,T] \times \R^d \to \R $ given by 
\begin{equation}\label{defH}
H(t,\mu; s',y';s,y) = \left(\tilde{\mL}^{y,t,\mu}-\mL^{t,\mu} \right)_{s',y'} \tp^y(t,\mu;s',y';s,y),
\end{equation}
where we recall that $\mL^{t,\mu}$ is the generator of \eqref{ttarget2}, such that the transition density $p$ of the flow $X$ defined by \eqref{ttarget2} writes:
\begin{equation}
p(t,\mu;s',y';s,y) = \tp^y(t,\mu;s',y';s,y) + \sum_{k =1}^{+\infty} \int_{s'}^s \int_{\R^d} H^{\otimes k}(t,\mu;r,u;s,y) \tilde{p}^{y'}(t,\mu;s',y';r,u) \d u \d r,
\end{equation}
 where $H^{\otimes k}$ is recursively defined by:
\begin{equation}\label{Hkpun}
H^{\otimes k+1}(t,\mu;s',y';s,y) = \int_{s'}^s \int_{\R^d} H^{\otimes k}(t,\mu;r,u;s,y)H^{}(t,\mu;s',y';r,u) \d u \d r,
\end{equation}
and $H^{\otimes 0} = {\rm Id}.$
\end{Proposition}

\begin{proof}
Let $(s,y)$ belong to $[0,T]\times \R^d$, the transition density $\tilde{p}^y(t,\mu;\cdot,\cdot;s,y)$ satisfies the Fokker-Planck equation:
\begin{equation*}\left\lbrace\begin{array}{ll}
\p_{s'} \tilde{p}^y(t,\mu;s',y';s,y) + \tilde{\mL}^{y,t,\mu}_{s',y'}  \tilde{p}^{y}(t,\mu;s',y';s,y)= 0,\quad (s',y') \in [0,s)\times  \R^d,\\
\tilde{p}^y(t,\mu;s,y';s,y)=\delta_y(y'),\end{array}\right.
\end{equation*}
which can be rewritten as
\begin{equation*}\left\lbrace\begin{array}{ll}
\p_{s'} \tilde{p}^y(t,\mu;s',y';s,y) + \mL^{t,\mu}_{s',y'}  \tilde{p}^y(t,\mu;s',y';s,y)= (\mL^{t,\mu}-\tilde{\mL}^{y,t,\mu})_{s',y'} \tilde{p}^y(t,\mu;s',y';s,y),\quad (s',y') \in [0,s)\times  \R^d\\
\tilde{p}^y(t,\mu;s',y';s,y)=\delta_y(y').\end{array}\right.
\end{equation*}
Note that $p(t,\mu;\cdot,\cdot;s,y)$ is a fundamental solution of this PDE. Therefore $\tilde{p}^y(t,\mu;\cdot,\cdot;s,y)$ writes, for all $(s',y')\in [0,s]\times \R^d$:

\begin{equation*}
\tilde{p}^y(t,\mu;s',y';s,y) = p(t,\mu;s',y';s,y) + \int_{s'}^s \int_{\R^d} (\mL^{t,\mu}-\tilde{\mL}^{y,t,\mu})_{r,u} \tilde{p}^y(t,\mu;r,u;s,y) p(t,\mu;s',y';r,u) \d u \d r.
\end{equation*}

Hence, by iterating $N$ times this procedure, we obtain that
\begin{eqnarray}\label{repexpinter}
p(t,\mu;s'y';s,y) &=&\tilde{p}^y(t,\mu;s',y';s,y) + \sum_{k =1}^N \int_{s'}^s \int_{\R^d} H^{\otimes k}(t,\mu;r,u;s,y) \tilde{p}^{y'}(t,\mu;s',y';r,u) \d u \d r\notag\\
&& \quad  +  \int_{s'}^s \int_{\R^d} H^{\otimes N+1}(t,\mu;r,u;s,y) p(t,\mu;s',y';r,u) \d u \d r.
\end{eqnarray}

In order to obtain the parametrix expansion of $p$, depending only on known quantities (\emph{i.e.} on the smoothing kernel $H$ defined by \eqref{defH} and on the transition density of the frozen process $\tilde{p}$) the idea consists in letting $N$ tend to infinity. To this aim, we need a ``good'' estimate on the approximation error. These controls are the estimate \eqref{gaussboundptilde} in Proposition \ref{prop:estifroedensity1} and the following Lemma.
\begin{lemme}\label{paramclassique} Under assumption $\mathbf{(H\mathcal{E}R)}$ the following assertion holds: there exists a positive constant $C_k^{\ref{paramclassique}}$ given by:
\begin{equation}\label{defck} 
C_k^{\ref{paramclassique}} = C_{\ref{paramclassique}}^k \prod_{r=1}^{k-1} \beta(\gamma_a r/2,\gamma_a/2),
\end{equation}
where $C_{\ref{paramclassique}}$ is a positive constant  depending only on known parameters in $\mathbf{(H\mathcal{E})}$, $\beta$ denotes the 
bêta-function and with the convention $\prod_1^0\equiv 1$, such that for all $s'<s$ in $[t,T]^2$ and $y',y$ in $\R^d$:
\begin{equation*}
|H^{\otimes k}(t,\mu;s',y';s,y)|\leq C_k^{\ref{paramclassique}} (s-s')^{k\gamma_a/2-1} \hat{p}_c(s',y';s,y).
\end{equation*}
\end{lemme}

On the one hand, the first term in the right hand side of \eqref{repexpinter} is controlled by a convolution of two Gaussian functions which is still Gaussian and it is clear from the asymptotic properties of the beta-function (that are recalled in Appendix \ref{APPC}) that the series converges. On the other hand, the $N^{{\rm th}}$ convolution of the kernel $H$ tends uniformly to 0 as $N$ tends to infinity (recall that $T$ is small) and since $p$ is a density, we deduce that the second term in the right hand side of \eqref{repexpinter} tends to 0. Therefore, the density $p$ writes:
\begin{eqnarray}\label{repexfinal}
p(\mu;t,x;s,y) &=&\tilde{p}^y(\mu;t,x;s,y) \\
&&+ \sum_{k \geq 1} \int_t^s \int_{\R^d} H^{\otimes k}(t,\mu;s',y';s,y) \tilde{p}^{y'}(\mu;t,x;s',y') \d y' \d s'.\notag
\end{eqnarray}
\end{proof}

\begin{proof}[Proof of Lemma \ref{paramclassique}.]
By using classical parametrix arguments (see Chapter 1 of \cite{friedman_partial_1964}) and by the  definition \eqref{defH} of $H$, we deduce that there exist two positive constants $C_{_{\ref{paramclassique}}} $ and $c$ depending only on known parameters in $\mathbf{(H\mathcal{E})}$ such that for all  $s' < s\in [t,T]^2$ and $y',y \in \R^d$:
\begin{equation}\label{esrte}
| H^{}(t,\mu;s',y';s,y) |\leq C_{_{\ref{paramclassique}}} (s-s')^{\gamma_a/2-1} \hat{p}_c(s',y';s,y).
\end{equation}

Suppose now as an induction hypothesis that for all $s' < s\in [t,T]^2$ and $y' \in \R^d$:
\begin{eqnarray}\label{FIH}
|H^{\otimes k}(t,\mu;s',y';s,y)| \leq C_k^{\ref{paramclassique}} (s-s')^{k\gamma_a/2-1} \hat{p}_c(s',y';s,y),
\end{eqnarray}
where $C_{k}^{\ref{paramclassique}}$ is defined by \eqref{defck}. Recall that for all integer $k$, $H^{\otimes k+1}$ is recursively defined by
\begin{equation}\label{rapHkpu}
H^{\otimes k+1}(t,\mu;s',y';s,y) = \int_{s'}^s \int_{\R^d} H^{\otimes k}(t,\mu;r,u;s,y)H^{}(t,\mu;s',y';r,u) \d u \d r.
\end{equation}

Hence, by plugging \eqref{esrte} and \eqref{FIH} in \eqref{rapHkpu} and using the Gaussian convolution we obtain that:
\begin{eqnarray*}
| H^{\otimes k+1}(t,\mu;s',y';s,y) |\leq C_k^{\ref{paramclassique}} C_{\ref{paramclassique}} \int_{s'}^s (s-r)^{\gamma_ak/2-1}(r-s')^{\gamma_a/2-1} \d r \hat{p}_c(s',y';s,y),
\end{eqnarray*}
and by the change of variable $r=(s-s')r'+s'$ we have:
\begin{eqnarray}\label{annexelaboundedeH}
| H^{\otimes k+1}(t,\mu;s',y';s,y) |& \leq&  C_{\ref{paramclassique}} C_k^{\ref{paramclassique}} (s-s')^{(k+1)\gamma_a/2-1} \int_0^1 (1-r')^{\bgamma_a k/2-1}(r')^{\gamma_a/2-1} \d r' \hat{p}_c(s',y';s,y)\notag\\
&=& C_{\ref{paramclassique}} C_k^{\ref{paramclassique}} (s-s')^{(k+1)\gamma_a/2-1} \beta(\gamma_a k/2,\gamma_a/2)\hat{p}_c(s',y';s,y)\notag\\
&=& C_{k+1}^{\ref{paramclassique}} (s-s')^{(k+1)\gamma_a/2-1}\hat{p}_c(s',y';s,y). \notag
\end{eqnarray}
\end{proof}

\subsection{Differentiation and estimation of the density $p$ along the measure}\label{subsec:estip}
We are now ready to prove Proposition \ref{prop:estidensity}. We have that 
\begin{equation*}
p(\mu;t,x;s,y) = \tp^y(\mu;t,x;s,y) + \sum_{k =1}^{+\infty} \int_t^s \int_{\R^d} H^{\otimes k}(t,\mu;s',y';s,y) \tilde{p}^{y'}(\mu;t,x;s',y') \d y' \d s',
\end{equation*}
so that the derivative of $p$ w.r.t. $\mu$ writes, at any point $z$ in $\R^d$:
\begin{eqnarray}\label{proof:estiderivmu1}
\p_\mu p(\mu;t,x;s,y)(z) &=&\p_\mu \tp^y(\mu;t,x;s,y)(z)\\
 && + \p_\mu \left(\sum_{k =1}^{+\infty} \int_t^s \int_{\R^d} H^{\otimes k}(t,\mu;s',y';s,y) \tilde{p}^{y'}(t,\mu;t,x;s',y') \d y' \d s'\right)(z).\notag
\end{eqnarray}

Then, in order to invert the integration and differentiation operators in the right hand side of the above equation, we have to show that for all $z$ in $\R^d$,
\begin{eqnarray}\label{proof:estiderivmu2}
(s',y')&\mapsto &\p_\mu H^{\otimes k}(t,\mu;s',y';s,y)(z) \tilde{p}^{y'}(\mu;t,x;s',y')\notag\\
&&+ H^{\otimes k}(t,\mu;s',y';s,y) \p_\mu\tilde{p}^{y'}(\mu;t,x;s',y')(z), 
\end{eqnarray}
is suitably bounded. More precisely, we have to obtain a Gaussian control on  the derivative of the $k$th iteration of the smoothing kernel $\p_\mu H^{\otimes k}$ and on $\p_\mu\tilde{p}$ so that the parametrix expansion still holds, in the same spirit of the proof of Proposition \ref{prop:parametrixrep}. These controls are given by the estimates on the frozen transition density in Proposition \ref{prop:estifroedensity1} and the following Lemma.

\begin{lemme}\label{EstiParametrixHmeasure}
Let $t$ in $[0,T]$ and let $\mu$ be the law of the process \eqref{ttarget1} at time $t$. For all positive integer $k$, there exists a positive constant $\tilde{C}_k^{\ref{EstiParametrixHmeasure}}$, depending only on known parameter in $\mathbf{(H\mathcal{E})}$ and recursively defined by:
\begin{eqnarray*}
\tilde{C}_{k}^{\ref{EstiParametrixHmeasure}}= (C_{k-1}^{\ref{paramclassique}}\tilde{C}+ C\tilde{C}_{k-1}^{\ref{EstiParametrixHmeasure}}) \beta((k-1)\bgamma_a/2,\bgamma_a/2),
\end{eqnarray*}
for $k\geq 2$ and $\tilde{C}_{1}^{\ref{EstiParametrixHmeasure}}=\tilde{C}$ and where $\bgamma_a=\gamma_a \wedge \gamma_a'$ such that
\begin{eqnarray*}
&&\left| \p_\mu H^{\otimes k}(t,\mu;s',y';s,y)(z) \right|\notag\\
&&\leq   \tilde{C}_k^{\ref{EstiParametrixHmeasure}}  (s'-s)^{k\bgamma_a/2-1}  \sum_{i=1}^2 (s'-t)^{(\alpha_i-1)/2} \overline{\p_\mu v}_{(\alpha_i-1)/2}^i(\mu;t,s')(z) \hat{p}_c(s',y';s,y),
\end{eqnarray*}
for all $s'<s\in (t,T]^2$, $y$, $y'$ in $\R^d$ and $z$ in $\R^d$.
\end{lemme}

From Lemma \ref{EstiParametrixHmeasure} and estimate \eqref{gaussboundptilde} in Proposition \ref{prop:estifroedensity1} we have that for all $k$:
\begin{eqnarray*}
&&\left|\p_\mu H^{\otimes k}(t,\mu;s',y';s,y)(z) \tilde{p}^{y'}(\mu;t,x;s',y') \right|\\
&\leq &\tilde{C}_k^{\ref{EstiParametrixHmeasure}} \bar{C}_{\ref{prop:estifroedensity1}}  (s-s')^{k\bgamma_a/2-1}  \sum_{i=1}^2 (s'-t)^{(\alpha_i-1)/2} \overline{\p_\mu v}_{(\alpha_i-1)/2}^i(\mu;t,s')(z)   \hat{p}_c(t,x;s,'y') \hat{p}_c(s',y';s,y),
\end{eqnarray*}
and from estimate \eqref{estipderivmu} of Proposition \ref{prop:estifroedensity1} and Lemma \ref{paramclassique} we have that for all $k$
\begin{eqnarray*}
&&\left| H^{\otimes k}(t,\mu;s',y';s,y) \p_\mu\tilde{p}^{y'}(\mu;t,x;s',y')(z) \right|\\
&\leq & C_k^{\ref{paramclassique}}  C_{\ref{prop:estifroedensity1}} (s-s')^{k\bgamma_a/2-1}  \sum_{i=1}^2 (s'-t)^{(\alpha_i-1)/2} \overline{\p_\mu v}_{(\alpha_i-1)/2}^i(\mu;t,s')(z)  \hat{p}_c(t,x;s',y') \hat{p}_c(s',y';s,y).
\end{eqnarray*}

We can hence  invert the differentiation and integration operators in the second term in the right hand side of \eqref{proof:estiderivmu1} and using property of Gaussian convolution we get that there exists a positive constant $C$, depending on known parameters in  $\mathbf{(H\mathcal{E})}$ only, such that

\begin{eqnarray*}
&&\p_\mu p(\mu;t,x;s,y)(z)  \\
&&\leq C (s-t)^{(\alpha_i-1)/2}  \bigg(C_{\ref{prop:estifroedensity1}} +\sum_{k\geq 1}  \Big\{(C_k^{\ref{paramclassique}}C_{\ref{prop:estifroedensity1}} +\tilde{C}_k^{\ref{EstiParametrixHmeasure}}\bar{C}_{\ref{prop:estifroedensity1}}) \beta(k\bar{\gamma_a}/2,(\alpha_i+1)/2)(s-t)^{k\bgamma_a/2}   \Big\}\bigg)\\
&& \qquad \times \hat{p}_c(t,x;s,y) \sum_{i=1}^2 \overline{\p_\mu v}_{(\alpha_i-1)/2}^i(\mu;t,s)(z),
\end{eqnarray*}
so that estimate \eqref{prop:estidensity2} of Proposition \ref{prop:estidensity} follows from the estimates on the parametrix constants $\tilde{C}_k^{\ref{EstiParametrixHmeasure}}$ and on the beta-function given in Appendix \ref{APPC}.\\

For the second assertion it is well seen from usual parametrix technique that 
\begin{equation*}
\p_x p(\mu;t,x;s,y) = \p_x\tp^y(\mu;t,x;s,y) + \sum_{k =1}^{+\infty} \int_t^s \int_{\R^d} H^{\otimes k}(t,\mu;s',y';s,y) \p_x\tilde{p}^{y'}(\mu;t,x;s',y') \d y' \d s',
\end{equation*}
so that the derivative of $\p_x p$ w.r.t. $\mu$ writes, at any point $z$ in $\R^d$:
\begin{eqnarray*}
\p_\mu (\p_x p)(\mu;t,x;s,y)(z) &=&\p_\mu (\p_x \tp^y)(\mu;t,x;s,y)(z)\\
 && + \p_\mu \left(\sum_{k =1}^{+\infty} \int_t^s \int_{\R^d} H^{\otimes k}(t,\mu;s',y';s,y) (\p_x\tilde{p}^{y'})(\mu;t,x;s',y') \d y' \d s'\right)(z).\notag
\end{eqnarray*}
We can use the same arguments as above with estimates \eqref{estipderivxmu}  instead of \eqref{estipderivmu} in Proposition \ref{prop:estifroedensity1} and we obtain that 

\begin{eqnarray*}
&&\p_\mu \p_x p(\mu;t,x;s,y)(z)  \\
&&\leq C (s-t)^{\alpha_i/2-1}  \bigg(C_{\ref{prop:estifroedensity1}}' +\sum_{k\geq 1}  \Big\{(C_k^{\ref{paramclassique}}C_{\ref{prop:estifroedensity1}}' +\tilde{C}_k^{\ref{EstiParametrixHmeasure}}\bar{C}_{\ref{prop:estifroedensity1}}) \beta(k\bar{\gamma_a}/2,(\alpha_i+1)/2)(s-t)^{k\bgamma_a/2}   \Big\}\bigg)\\
&& \qquad \times \hat{p}_c(t,x;s,y) \sum_{i=1}^2 \overline{\p_\mu v}_{(\alpha_i-1)/2}^i(\mu;t,s)(z),
\end{eqnarray*}
from which we deduce estimate \eqref{prop:estidensity3} of Proposition \ref{prop:estidensity}.

\appendix
\section{Proofs of Lemmas \ref{EstiParametrixHmeasure} and \ref{prop:estifroedensity1}}\label{sec:prooflemmas}
In order to avoid heavy notations, the proofs are done in the real case ($d=1$). We also recall that $T<1$.

\begin{proof}[Proof of Lemma \ref{EstiParametrixHmeasure}]
Recall that by definition
\begin{eqnarray*}
H(t,\mu;s',y';s,y) &=& \left(b(s',y,v_{s'}(t,\mu,\varphi_1)) - b(s',y',v_{s'}(t,\mu,\varphi_1))\right) \p_x \tp(\mu;s',y';s,y) \\
&&+ \frac{1}{2}\left[(a(s',y,v_{s'}(t,\mu,\varphi_2)) - a(s,y',v_{s'}(t,\mu,\varphi_2)))\p_x^2 \tp(\mu;s',y';s,y)\right],
\end{eqnarray*}
so that, for any $z$ in $\R^d$:

\begin{eqnarray*}
&&\p_\mu H(t,\mu;s',y';s,y)(z) \\
&&= (b(s',y,v_{s'}(t,\mu,\varphi_1)) - b(s',y',v_{s'}(t,\mu,\varphi_1)))\p_\mu\p_x \tp(\mu;s',y';s,y)(z) \\
&& \quad + (\p_3 b(s',y,v_{s'}(t,\mu,\varphi_1)) - \p_3 b(s',y',v_{s'}(t,\mu,\varphi_1)))\p_\mu v_{s'}(t,\mu,\varphi_1)(z) \p_x \tp(\mu;s',y';s,y)\\
&&\quad +  \frac{1}{2}(a(s',y,v_{s'}(t,\mu,\varphi_2)) - a(s',y',v_{s'}(t,\mu,\varphi_2)))\p_\mu \p_x^2 \tp(\mu;s',y';s,y)(z)\\
&&\quad +  \frac{1}{2}(\p_3 a(s',y,v_s(t,\mu,\varphi_2)) - \p_3 a(s',y',v_{s'}(t,\mu,\varphi_2)))\p_\mu v_{s'}(t,\mu,\varphi_2)(z) \p_x^2 \tp(\mu;s',y';s,y).
\end{eqnarray*}

Hence
\begin{eqnarray*}
&&\left|\p_\mu H(t,\mu;s',y';s,y)(z)\right| \\
&& \leq   C\Bigg\{||b||_{\infty}\sum_{i=1}^2 (s'-t)^{(\alpha_i-1)/2}\overline{\p_\mu v}_{(\alpha_i-1)/2}^i(\mu;t,s')(z)  (s-s')^{-1/2} \\
&& + ||\p_3 b||_{\infty}  (s'-t)^{(\alpha_1-1)/2} \overline{\p_\mu v}_{(\alpha_1-1)/2}^1(\mu;t,s')(z) (s-s')^{-1/2} \\
&& + ||a||_{\bgamma_a} |y-y'|^{\bgamma_a} \sum_{i=1}^2 (s'-t)^{(\alpha_i-1)/2}\overline{\p_\mu v}_{(\alpha_i-1)/2}^i(\mu;t,s')(z) (s-s')^{-1} \\
&& + ||\p_3 a||_{\gamma_{a}'} |y-y'|^{\gamma_{a}'}  (s'-t)^{(\alpha_2-1)/2} \overline{\p_\mu v}_{(\alpha_2-1)/2}^2(\mu;t,s')(z) (s-s')^{-1} \Bigg\}\hat{p}_c(s',y';s,y),
\end{eqnarray*}
where $\bgamma_a = \gamma_a\wedge \gamma_a'$. Therefore, by using the Gaussian decay of $\hat{p}_c$:

\begin{eqnarray*}
&&\left|\p_\mu H(t,\mu;s',y';s,y)(z)\right| \\
&& \leq   C\Bigg\{||b||_{\infty}  \sum_{i=1}^2 (s'-t)^{(\alpha_i-1)/2} \sum_{i=1}^2 \overline{\p_\mu v}_{(\alpha_i-1)/2}^i(\mu;t,s')(z)  (s-s')^{-1/2} \\
&& + ||\p_3 b||_{\infty}   (s'-t)^{(\alpha_1-1)/2} \overline{\p_\mu v}_{(\alpha_1-1)/2}^1(\mu;t,s')(z) (s-s')^{-1/2} \\
&& + ||a||_{\bgamma_a} \sum_{i=1}^2(s'-t)^{(\alpha_i-1)/2} \overline{\p_\mu v}_{(\alpha_i-1)/2}^i(\mu;t,s')(z) (s-s')^{-1+\bgamma_a/2} \\
&& + ||\p_3 a||_{\gamma_{a}'}   (s'-t)^{(\alpha_2-1)/2} \overline{\p_\mu v}_{(\alpha_2-1)/2}^2(\mu;t,s')(z) (s-s')^{-1+\gamma_{a}'/2} \Bigg\}\hat{p}_c(s',y';s,y).
\end{eqnarray*}

So, there exists a positive constant $\tilde{C}$ depending n known parameters in \textbf{(H$\mathbf{\mathcal{E}}$R)} such that:
\begin{eqnarray}\label{eq:estidepreuvelemme}
&&\left|\p_\mu H(t,\mu;s',y';s,y)(z)\right|  \\
&& \leq  \tilde{C} \sum_{i=1}^2 (s'-t)^{(\alpha_i-1)/2}\overline{\p_\mu v}_{(\alpha_i-1)/2}^i(\mu;t,s')(z) (s-s')^{\bgamma_a/2-1} \hat{p}_c(s',y';s,y).\notag
\end{eqnarray}

Assume now as an induction hypothesis that for all $(s',y')$ in $(t;s)\times \R^d$:
\begin{eqnarray*}
&&\left|\p_\mu H^{\otimes k}(t,\mu;s',y';s,y)(z)\right| \leq \tilde{C}_k^{\ref{EstiParametrixHmeasure}} \sum_{i=1}^2 (s'-t)^{(\alpha_i-1)/2}\overline{\p_\mu v}_{(\alpha_i-1)/2}^i(\mu;t,s')(z) (s-s')^{k\bgamma_a/2-1} \hat{p}_c(s',y';s,y),
\end{eqnarray*}
where
\begin{equation*}
 \tilde{C}_k^{\ref{EstiParametrixHmeasure}}  = (C_{k-1}^{\ref{EstiParametrixHmeasure}}\tilde{C}+ C\tilde{C}_{k-1}^{\ref{EstiParametrixHmeasure}}) \beta((k-1)\bgamma_a/2,\bgamma_a/2).
\end{equation*}

We then have
\begin{eqnarray}
&&\p_\mu H^{\otimes k+1}(t,\mu;s',y';s,y)(z)\label{tobound}\\
&&= \int_{s'}^s \int_{\R^d} \p_\mu H^{\otimes k}(t,\mu;r,u;s,y)(z)H^{}(t,\mu;s',y';r,u) \d u \d r\notag\\
&& \quad +  \int_{s'}^{s} \int_{\R^d} H^{\otimes k}(t,\mu;r,u;s,y)\p_\mu H^{}(t,\mu;s',y';r,u)(z) \d u\d r.\notag
\end{eqnarray}
We can bound the first term in the right hand side by using the induction hypothesis above, the estimation \eqref{rapHkpu} on $H$ and the property of the Gaussian convolution:

\begin{eqnarray*}
&&\left|\int_{s'}^s \int_{\R^d} \p_\mu H^{\otimes k}(t,\mu;r,u;s,y)(z)H^{}(t,\mu;s',y';r,u) \d u \d r\right|\\ 
&&\leq  C_{\ref{paramclassique}}\tilde{C}_k^{\ref{EstiParametrixHmeasure}} \int_{s'}^s  \sum_{i=1}^2 (s'-t)^{(\alpha_i-1)/2}\overline{\p_\mu v}_{(\alpha_i-1)/2}^i(\mu;t,r)(z)   (s-r)^{k\bgamma_a/2-1} (r-s')^{\bgamma_a/2-1} \d r \hat{p}_c(s',y';s,y)\\
&&\leq  C_{\ref{paramclassique}}\tilde{C}_k^{\ref{EstiParametrixHmeasure}} \sum_{i=1}^2 (s'-t)^{(\alpha_i-1)/2}\overline{\p_\mu v}_{(\alpha_i-1)/2}^i(\mu;t,s')(z) \hat{p}_c(s',y';s,y) \int_{s'}^s    (s-r)^{k\bgamma_a/2-1}(r-s')^{\bgamma_a/2-1}  \d r.
\end{eqnarray*}
By the change of variable $r=(s-s')r'+s'$, one can show that

\begin{equation*}
\int_{s'}^s    (s-r)^{k\bgamma_a/2-1}(r-s')^{\bgamma_a/2-1}  \d r \leq (s-s')^{(k+1)\bgamma_a/2-1} \beta(k\bgamma_a/2,\bgamma_a/2),
\end{equation*}
so that

\begin{eqnarray}
&&\left|\int_{s'}^s \int_{\R^d} \p_\mu H^{\otimes k}(t,\mu;r,u;s,y)(z)H^{}(t,\mu;s',y';r,u) \d u \d r\right|\label{tobound1}\\
&& \leq  C_{\ref{paramclassique}}\tilde{C}_k^{\ref{EstiParametrixHmeasure}} \sum_{i=1}^2 \overline{\p_\mu v}_{(\alpha_i-1)/2}^i(\mu;t,s') (z) \hat{p}_c(s',y';s,y) (s-s')^{(k+1)\bgamma_a/2-1} \beta(k\bgamma_a/2,\bgamma_a/2)\notag.
\end{eqnarray}

Now, we bound the second term in the right hand side of \eqref{tobound}. We have from Lemma \ref{paramclassique}:
\begin{equation*}
\left|H^{\otimes k}(t,\mu;r,u;s,y)\right|\leq C_k^{\ref{paramclassique}} (s-r)^{k\bgamma_a/2-1} \hat{p}_c(r,u;s,y),
\end{equation*}
where

\begin{equation*}
C_k^{\ref{paramclassique}} = C_{\ref{paramclassique}}^k \prod_{l=1}^{k-1}\beta(r\bgamma_a/2,\bgamma_a/2).
\end{equation*}
So that, thanks to estimate \eqref{eq:estidepreuvelemme}:

\begin{eqnarray}
&&\left|\int_{s'}^{s} \int_{\R^d} H^{\otimes k}(t,\mu;r,u;s,y)\p_\mu H^{}(t,\mu;s',y';r,u)(z) \d u\d r\right| \label{tobound2}\\
&&\quad \leq  \sum_{i=1}^2(s'-t)^{(\alpha_i-1)/2} \overline{\p_\mu v}_{(\alpha_i-1)/2}^i(\mu;t,s')(z) C_k^{\ref{paramclassique}}\tilde{C} \int_{s'}^s    (s-r)^{k\bgamma_a/2-1}(r-s')^{\bgamma_a/2-1}  \d r\notag.
\end{eqnarray}

Hence, by plugging \eqref{tobound1} and \eqref{tobound2} in \eqref{tobound} we get:
\begin{eqnarray*}
&&\left|\p_\mu H^{\otimes k+1}(t,\mu;s',y';s,y)(z) \right|\\
&&\leq (s-s')^{(k+1)\bgamma_a/2-1}  \sum_{i=1}^2 (s'-t)^{(\alpha_i-1)/2}\overline{\p_\mu v}_{(\alpha_i-1)/2}^i(\mu;t,s')(z)  (C_k^{\ref{paramclassique}}\tilde{C}+ C\tilde{C}_k^{\ref{EstiParametrixHmeasure}}) \beta(k\bgamma_a/2,\bgamma_a/2),
\end{eqnarray*}
and the induction is true since:

\begin{equation*}
\tilde{C}_{k+1}^{\ref{EstiParametrixHmeasure}} = (C_k^{\ref{paramclassique}}\tilde{C}+ C\tilde{C}_k^{\ref{EstiParametrixHmeasure}}) \beta(k\bgamma_a/2,\bgamma_a/2).\\
\end{equation*}
\end{proof}

\begin{proof}[Proof of Lemma \ref{prop:estifroedensity1}.]
We begin with the following Claim.
\begin{claim}\label{EstiDerivMoyVar}
The following estimates hold:
\begin{enumerate}[$\bullet$]
\item $\displaystyle \big|\p_\mu m_{t,s}^{\xi}(t,\mu)(z)\big|\leq (s-t)^{(\alpha_1+1)/2}   ||\p_3b||_{\infty} \overline{\p_\mu v}^1_{(1-\alpha_1)/2}(\mu;t,s)(z)$,
\item $\displaystyle \big|\p_\mu a_{t,s}^{\xi}(t,\mu)(z)\big| \leq  (s-t)^{(\alpha_2+1)/2}  ||\p_3a||_{\infty} \overline{\p_\mu v}^2_{(1-\alpha_2)/2}(\mu;t,s)(z) $,
\end{enumerate}
and for all $s' \in (t,T]$:
\begin{enumerate}[$\bullet$]
\item $\displaystyle \big|\p_\mu m_{s',s}^{\xi}(t,\mu)(z)\big|\leq (s-s')(s'-t)^{(\alpha_1-1)/2}   ||\p_3b||_{\infty} \overline{\p_\mu v}^1_{(1-\alpha_1)/2}(\mu;t,s)(z)$,
\item $\displaystyle \big|\p_\mu a_{t,s}^{\xi}(t,\mu)(z)\big| \leq  (s-s')(s'-t)^{(\alpha_2-1)/2}  ||\p_3a||_{\infty} \overline{\p_\mu v}^2_{(1-\alpha_2)/2}(\mu;t,s)(z) $.
\end{enumerate}
\end{claim}
We also recall the classical estimate coming from the uniform ellipticity of $a$: there exists $\Lambda>0$ such that, for all positive $\gamma$:
\begin{equation}
\frac{1}{\big[a_{s',s}^{\xi}(t,\mu)\big]^{\gamma}} < \Lambda^{-\gamma} (s-s')^{-\gamma}.
\end{equation}

The derivative of $\tp$ evaluated at any point $z$ in $\R$ is then given by:
\begin{eqnarray*}
&&\p_\mu \tp^{\xi}(t,\mu;s',y';s,y)(z) \\
&& = \Bigg(-\frac{1}{2} \frac{\p_\mu  a_{s',s}^{\xi}(t,\mu)(z)}{\big[a_{s',s}^{\xi}(t,\mu)\big]^{}} + \Bigg(\frac{1}{2}\frac{\big(\p_\mu  a_{s',s}^{\xi}(t,\mu)(z)\big)\big(y-y'-m_{s',s}^{\xi}(t,\mu)\big)}{\big[a_{s',s}^{\xi}(t,\mu)\big]^{3/2}} + \frac{\p_\mu m_{s',s}^{\xi}(t,\mu)(z)}{\big[a_{s',s}^{\xi}(t,\mu)\big]^{1/2}}\Bigg)\\
&&\quad \times  \Bigg(\frac{y-y'-m_{s',s}^{\xi}(t,\mu)}{\big[a_{t,s}^{\xi}([X^{t,\mu}])\big]^{1/2}}\Bigg)\Bigg) \frac{1}{\sqrt{2\pi}}\frac{1}{\big[a_{s',s}^{\xi}(t,\mu)\big]^{1/2}} \exp\l(-\frac{1}{2}\bigg|\big[a_{s',s}^{\xi}(t,\mu)\big]^{-1/2}\big(y-y'-m_{s',s}^{\xi}(t,\mu)\big)\bigg|^2 \r).
\end{eqnarray*}
Now, by using the Gaussian decay of $\tp$, and estimates of Claim \ref{EstiDerivMoyVar}, we obtain that, when $s'=t$:

\begin{eqnarray*}
&&\big|\p_\mu \tp^{\xi}(t,\mu;t,y';s,y)(z)\big| \\
&& \leq C\bigg(\Lambda^{-1}||\p_3a||_{\infty} \sup_{r\in[t,s]}\overline{\p_\mu v}^2_{(1-\alpha_2)/2}(\mu;t,s)(z) (s-t)^{(\alpha_2-1)/2}\\
&&\quad + \Lambda^{-1}||\p_3b||_{\infty} \sup_{r\in[t,s]}\overline{\p_\mu v}^1_{(1-\alpha_1)/2}(\mu;t,s)(z) (s-t)^{(\alpha_1-1)/2}\bigg)\\
&& \qquad \frac{1}{\sqrt{2\pi}}\frac{1}{\big[a_{t,s}^{\xi}(t,\mu)\big]^{1/2}} \exp\l(-c\frac{1}{2}\bigg|\big[a_{t,s}^{\xi}(t,\mu)\big]^{-1/2}\big(y-y'-m_{t,s}^{\xi}(t,\mu)\big)\bigg|^2 \r),
\end{eqnarray*}
for some positive constant $C$ and $c$, with $c$ strictly less than $1$. We now compute the space derivatives: for all $s'<s$ in $[t,T]^2$

\begin{eqnarray*}
&&\p_x \tp^{\xi}(t,\mu;s',y';s,y) \\
&&=  \frac{\big(y-y'-m_{s',s}^{\xi}(t,\mu)\big)}{\big[a_{s',s}^{\xi}(t,\mu)\big]} \frac{1}{\sqrt{2\pi }}\frac{1}{\big[a_{s',s}^{\xi}(t,\mu)\big]^{1/2}} \exp\l(-c\frac{1}{2}\bigg|\big[a_{s',s}^{\xi}(t,\mu)\big]^{-1/2}\big(y-y'-m_{s',s}^{\xi}(t,\mu)\big)^*\bigg|^2 \r),
\end{eqnarray*}
which gives

\begin{eqnarray*}
|\p_x \tp^{\xi}(t,\mu;s',y';s,y)| \leq   C (s-t)^{-1/2} \frac{1}{\sqrt{2\pi }}\frac{1}{\big[a_{s',s}^{\xi}(t,\mu)\big]^{1/2}} \exp\l(-\frac{1}{2}\bigg|\big[a_{s',s}^{\xi}(t,\mu)\big]^{-1/2}\big(y-y'-m_{s',s}^{\xi}(t,\mu)\big)^*\bigg|^2 \r).
\end{eqnarray*}

Next

\begin{eqnarray*}
\p^2_{x} \tp^{\xi}(t,\mu;s',x;s,y) &=&  \left(\frac{\big(y-y'-m_{s',s}^{\xi}(t,\mu)\big)^2}{\big[a_{s',s}^{\xi}(t,\mu)\big]^2} - \frac{1}{\big[a_{s',s}^{\xi}(t,\mu)\big]}\right)\\
&&\times \frac{1}{\sqrt{2\pi }}\frac{1}{\big[a_{s',s}^{\xi}(t,\mu)\big]^{1/2}} \exp\l(-\frac{1}{2}\bigg|\big[a_{s',s}^{\xi}(t,\mu)\big]^{-1/2}\big(y-y'-m_{s',s}^{\xi}(t,\mu)\big)^*\bigg|^2 \r),
\end{eqnarray*}
so that

\begin{eqnarray*}
&&\p^2_{x} \tp^{\xi}(t,\mu;s',y';s,y) \\
&&\leq  C(1 + \Lambda^{-1}) (s-s')^{-1} \frac{1}{2\pi } \frac{1}{\big[a_{s',s}^{\xi}(t,\mu)\big]^{1/2}} \exp\l(-\frac{1}{2}\bigg|\big[a_{s',s}^{\xi}(t,\mu)\big]^{-1/2}\big(y-y'-m_{s',s}^{\xi}(t,\mu)\big)^*\bigg|^2 \r).
\end{eqnarray*}

Concerning the cross derivatives, we have, at any point $z$ of $\R$:
\begin{eqnarray*}
&&\p_\mu\p_x \tp^{\xi}(t,\mu;s',y';s,y)(z) \\
&&=  \Bigg\{-\frac{\big(\p_\mu m_{s',s}^{\xi}(t,\mu)(z)\big)}{\big[a_{s',s}^{\xi}(t,\mu)\big]^{}} - \frac{3}{2}\frac{\big(\p_\mu a_{s',s}^{\xi}(t,\mu)(z)\big)\big(y-y'-m_{s',s}^{\xi}(t,\mu)\big)}{\big[a_{s',s}^{\xi}(t,\mu)\big]^{2}}\\
&&\quad  -\frac{1}{2} \frac{\p_\mu  a_{s',s}^{\xi}(t,\mu)(z)}{\big[a_{s',s}^{\xi}(t,\mu)\big]^{}} +  \Bigg(\frac{1}{2}\frac{\big(\p_\mu  a_{s',s}^{\xi}(t,\mu)(z)\big)\big(y-x-m_{t,s}^{\xi}([X^{t,\mu}])\big)}{\big[a_{t,s}^{\xi}([X^{t,\mu}])\big]^{3/2}} + \frac{\p_\mu m_{s',s}^{\xi}(t,\mu)(z)}{\big[a_{s',s}^{\xi}(t,\mu)\big]^{1/2}}\Bigg)\\
&&\quad \times  \Bigg(\frac{y-y'-m_{s',s}^{\xi}(t,\mu)}{\big[a_{s',s}^{\xi}(t,\mu)\big]^{1/2}}\Bigg)\Bigg)\Bigg\}\\
&& \qquad \frac{1}{\sqrt{2\pi }}\frac{1}{\big[a_{s',s}^{\xi}(t,\mu)\big]^{1/2}}\exp\l(-\frac{1}{2}\bigg|\big[a_{s',s}^{\xi}(t,\mu)\big]^{-1/2}\big(y-y'-m_{s',s}^{\xi}(t,\mu)\big)^*\bigg|^2 \r).
\end{eqnarray*}
Thus we have: when $s'=t$, 

\begin{eqnarray*}
&&|\p_\mu\p_x \tp^{\xi}(t,\mu;t,y';s,y)(z)| \\
&&\leq  \Big(\bigg((s-t)^{(\alpha_1-1)/2}+ (s-t)^{\alpha_1}\Big)||\p_3b||_{\infty} \overline{\p_\mu v}^1_{(1-\alpha_1)/2}(\mu;t,s)(z)\\
&& \quad +\Big((s-t)^{\alpha_2/2-1} + 2(s-t)^{(\alpha_2-1)/2} \Big) ||\p_3a||_{\infty}\overline{\p_\mu v}^2_{(1-\alpha_2)/2}(\mu;t,s)(z)\Bigg)\\
&&\qquad \times \frac{C\Lambda^{-1}}{\big[a_{t,s}^{\xi}([X^{t,\mu}])\big]^{1/2}}\exp\l(-c\frac{1}{2}\bigg|\big[a_{t,s}^{\xi}([X^{t,\mu}])\big]^{-1/2}\big(y-y'-m_{t,s}^{\xi}([X^{t,\mu}])\big)^*\bigg|^2 \r),
\end{eqnarray*}
and when $s'>t$,
\begin{eqnarray*}
&&|\p_\mu\p_x \tp^{\xi}(t,\mu;s',y';s,y)(z)| \\
&&\leq  \Bigg(\Big((s'-t)^{(\alpha_1-1)/2}+  (s'-t)^{(\alpha_1-1)/2}\Big)||\p_3b||_{\infty} \overline{\p_\mu v}^1_{(1-\alpha_1)/2}(\mu;t,s)(z)\\
&& \quad + \Big((s'-t)^{(\alpha_2-1)/2} + 2(s'-t)^{(\alpha_2-1)/2} \Big)||\p_3a||_{\infty}\overline{\p_\mu v}^2_{(\alpha_2-1)/2}(\mu;t,s)(z)\Bigg)\\
&&\qquad \times \frac{C(s-s')^{-(1/2)}}{\big[a_{s',s}^{\xi}(t,\mu)\big]^{1/2}}\exp\l(-c\frac{1}{2}\bigg|\big[a_{s',s}^{\xi}(t,\mu)\big]^{-1/2}\big(y-y'-m_{s',s}^{\xi}(t,\mu)\big)^*\bigg|^2 \r).
\end{eqnarray*}
We conclude with:

\begin{eqnarray*}
&&\p_\mu \p^2_{x} \tp^{\xi}(t,\mu;s',y';s,y)(z) \\
&&=  \Bigg\{\Bigg(\frac{2\p_\mu m_{s',s}^\xi([X^{t,\mu}])(z)(y-y'-m_{s',s}^\xi([X^{t,\mu}]))}{\big[a_{s',s}^{\xi}(t,\mu)\big]^{2}} - \frac{2\big(y-y'-m_{s',s}^{\xi}(t,\mu)\big)^2\p_\mu a_{s',s}^{\xi}(t,\mu)(z) }{\big[a_{s',s}^{\xi}(t,\mu)\big]^{3}} \\
&&+ \frac{\p_\mu a_{s',s}^{\xi}(t,\mu)(z)}{\big[a_{s',s}^{\xi}(t,\mu)\big]^{2}}\Bigg) +\left(\frac{\big(y-y'-m_{s',s}^{\xi}(t,\mu)\big)^2}{\big[a_{s',s}^{\xi}(t,\mu)\big]^{3/2}} - \frac{1}{\big[a_{s',s}^{\xi}(t,\mu)\big]^{1/2}}\right)\left(-\frac{\p_\mu a_{s',s}^{\xi}(t,\mu)(z)}{2\big[a_{s',s}^{\xi}(t,\mu)\big]^{3/2}}\right) \\
&&+\left(\frac{\big(y-y'-m_{s',s}^{\xi}(t,\mu)\big)^2}{\big[a_{s',s}^{\xi}(t,\mu)\big]^2} - \frac{1}{\big[a_{s',s}^{\xi}(t,\mu)\big]}\right)\\
&&\quad \times \Bigg(-\frac{1}{2} \frac{\p_\mu  a_{s',s}^{\xi}(t,\mu)(z)}{\big[a_{s',s}^{\xi}(t,\mu)\big]^{}} + \Bigg(\frac{1}{2}\frac{\big(\p_\mu  a_{s',s}^{\xi}(t,\mu)(z)\big)\big(y-y'-m_{s',s}^{\xi}(t,\mu)\big)}{\big[a_{s',s}^{\xi}(t,\mu)\big]^{3/2}} + \frac{\p_\mu m_{s',s}^{\xi}(t,\mu)(z)}{\big[a_{s',s}^{\xi}(t,\mu)\big]^{1/2}}\bigg) \\
&&\qquad \times \bigg(\frac{y-y'-m_{s',s}^{\xi}(t,\mu)}{\big[a_{s',s}^{\xi}(t,\mu)\big]^{1/2}}\Bigg)\Bigg)\Bigg\}\\
&&\times \frac{1}{2\pi } \frac{1}{\big[a_{s',s}^{\xi}(t,\mu)\big]^{1/2}}  \exp\l(-\frac{1}{2}\bigg|\big[a_{s',s}^{\xi}(t,\mu)\big]^{-1/2}\big(y-y'-m_{s',s}^{\xi}(t,\mu)\big)^*\bigg|^2 \r).
\end{eqnarray*}

Which gives, for all $s'>t$,
\begin{eqnarray*}
&&\left|\p_\mu \p^2_{x} \tp^{\xi}(\mu;t,x;s,y)(z)\right| \\
&& \leq \Bigg(||\p_3b||_{\infty}(s-s')^{-1/2}(s'-t)^{(\alpha_1-1)/2}  \overline{\p_\mu v}^1_{(1-\alpha_1)/2}(\mu;t,s)(z)\\
&&+ ||\p_3a||_{\infty}\Big((s-s')^{-1} +(s-s')^{-1/2} \Big)(s'-t)^{(\alpha_2-1)/2} \overline{\p_\mu v}^2_{(1-\alpha_2)/2}(\mu;t,s)(z)\Bigg)\\
&&\times\frac{1}{\sqrt{2\pi }} \frac{1}{\big[a_{s',s}^{\xi}(t,\mu)\big]^{1/2}}  \exp\l(-\frac{1}{2}\bigg|\big[a_{s',s}^{\xi}(t,\mu)\big]^{-1/2}\big(y-y'-m_{s',s}^{\xi}(t,\mu)\big)^*\bigg|^2 \r).
\end{eqnarray*}

\end{proof}

\begin{proof}[Proof of Claim \ref{EstiDerivMoyVar}]
From chain rule we have, for all $\xi$ in $\R$ and for all $s'<s \in [t,T]$:
\begin{eqnarray*}
\p_\mu m_{s',s}^{\xi}(t,\mu)(z) &=& \bigg(\p_\mu \int_{s'}^s b(r,\xi,\langle \varphi_1,[X_r^{t,\mu})]\rangle) \d r\bigg)(z) \\
&=&\int_{s'}^s \p_3b(r,\xi,\langle \varphi_1,[X_r^{t,\mu})]\rangle)\p_\mu v_r(t,\mu,\varphi_1)(z) \d r, 
\end{eqnarray*}
for all $z$ in $\R$ and where we recall that $v$ is defined by \eqref{def:v}. So,
\begin{eqnarray*}
\big|\p_\mu m_{s',s}^{\xi}(t,\mu)(z)\big| &\leq& ||\p_3 b||_{\infty} \sup_{r \in [t,s]}\big\{(r-t)^{(1-\alpha_1)/2}|\p_\mu v_r(t,\mu,\varphi_1)(z)|\big\} \int_{s'}^s  (r-t)^{(\alpha_1-1)/2}   \d r.
\end{eqnarray*}
Hence, for all $z$ in $\R$
\begin{eqnarray*}
\big|\p_\mu m_{s',s}^{\xi}(t,\mu)(z)\big| &\leq & ||\p_3 b||_{\infty} \sup_{r \in [t,s]}\big\{(r-t)^{(1-\alpha_1)/2}|\p_\mu v_r(t,\mu,\varphi_1)(z)|\big\}  (s-t)^{(\alpha_1+1)/2},
\end{eqnarray*}
when $s'=t$ and 
\begin{eqnarray*}
\big|\p_\mu m_{s',s}^{\xi}(t,\mu)(z)\big| &\leq & ||\p_3 b||_{\infty} \sup_{r \in [t,s]}\big\{(r-t)^{(1-\alpha_1)/2}|\p_\mu v_r(t,\mu,\varphi_1)(z)|\big\} (s'-s) (s'-t)^{(\alpha_1-1)/2},
\end{eqnarray*}
when $s'>t$. Next we have,
\begin{eqnarray*}
\p_\mu a_{s',s}^{\xi}(t,\mu)(z) &=& \bigg(\p_\mu \int_{s'}^s a(r,\xi,\langle \varphi_2,[X_r^{t,\mu})]\rangle) \d r\bigg)(z) \\
&=&\int_{s'}^s \p_3a(r,\xi,\langle \varphi_2,[X_r^{t,\mu})]\rangle)\p_\mu v_r(t,\mu,\varphi_2)(z) \d r, 
\end{eqnarray*}
for all $z$ in $\R$. Therefore
\begin{eqnarray*}
\big|\p_\mu a_{s',s}^{\xi}(t,\mu)(z) \big| &\leq& ||\p_3a ||_{\infty} \sup_{r \in [t,s]}\big\{(r-t)^{(1-\alpha_2)/2}|\p_\mu v_r(t,\mu,\varphi_2)(z)|\big\} \int_t^s  (r-t)^{(\alpha_2-1)/2}   \d r.
\end{eqnarray*}
So
\begin{eqnarray*}
\big|\p_\mu a_{s',s}^{\xi}(t,\mu)(z) \big| &\leq & ||\p_3 a||_{\infty} \sup_{r \in [t,s]}\big\{(r-t)^{(1-\alpha_2)/2}|\p_\mu v_r(t,\mu,\varphi_2)(z)|\big\}  (s-t)^{(\alpha_2+1)/2},
\end{eqnarray*}
when $s'=t$ and
\begin{eqnarray*}
\big|\p_\mu a_{s',s}^{\xi}(t,\mu)(z) \big| &\leq & ||\p_3 a||_{\infty} \sup_{r \in [t,s]}\big\{(r-t)^{(1-\alpha_2)/2}|\p_\mu v_r(t,\mu,\varphi_2)(z)|\big\}  (s'-s)(s'-t)^{(\alpha_2-1)/2},
\end{eqnarray*}
when $s'>t$.
\end{proof}

\section{Asymptotic properties of the parametrix constants}\label{APPC}
\begin{claim}\label{asymptobeta}
There exists a strictly finite and strictly positive integer $K(\gamma_a)$  such that for all $k \geq K(\gamma_a)$:
$$C_k^{\ref{paramclassique}} = C(K(\gamma_a))C_{\ref{paramclassique}}^k\frac{4^{k}}{\gamma_a^k(k!)^{\gamma_a/2}},$$
where $C(K(\gamma_a)) = 4^{-K(\gamma_a)}(K(\gamma_a)!)^{\gamma_a/2}\prod_{l=1}^{K(\gamma_a)-1}\beta(l\gamma_a/2,\gamma_a/2).$
\end{claim}
\begin{proof}
Let $K(\gamma_a)$ be equal to $\lceil 2/\gamma_a \rceil$, so that for all $k \geq K(\gamma_a)$
$$k\gamma_a/2-1\geq 0,$$ 
and recall that

$$\beta(\gamma_a k/2,\gamma_a/2) = \int_0^1(1-s)^{k\gamma_a/2}  s^{\gamma_a/2-1} \d s.$$
For all positive $\epsilon$ strictly less than 1, we have, for all $k\geq K(\gamma_a)$:

\begin{eqnarray}\label{calculint}
\int_0^1 (1-s)^{\gamma_a k/2-1}s^{\gamma_a/2-1} \d s &=& \int_0^1 (1-s)^{\gamma_a/2-1}s^{\gamma_a k/2-1} \d s \\
&=& \int_0^{1-\epsilon} (1-s)^{\gamma_a/2-1}s^{\gamma_a k/2-1} \d s + \int_{1-\epsilon}^1 (1-s)^{\gamma_a/2-1}s^{\gamma_a k/2-1} \d u\notag\\
&\leq & \frac{1}{\epsilon^{1-\gamma_a/2}} \int_0^{1-\epsilon} s^{\gamma_a k/2-1} \d s + \int_{1-\epsilon}^1 (1-s)^{\gamma_a/2-1}\d s\notag\\
&\leq & \frac{1}{\epsilon^{1-\gamma_a/2}}  \frac{2}{k\gamma_a} + \frac{2}{\gamma_a}\epsilon^{\gamma_a/2}.\notag
\end{eqnarray}
So that, by letting $\epsilon= 1/k$ we have
\begin{eqnarray}\label{annexeboundepsi}
\int_0^1 (1-s)^{k/2-1}s^{-1/2} \d s\leq \frac{4}{\gamma_ak^{\gamma_a/2}},
\end{eqnarray}
which gives the desired result.
\end{proof}

\begin{claim}
There exists a strictly finite and strictly positive integer $K : = K(\gamma_a,\gamma_a')$  such that for all $k \geq K$:
\begin{equation*}
\tilde{C}_{k+1}^{\ref{EstiParametrixHmeasure}} = (C_k^{\ref{paramclassique}}\tilde{C}+ C\tilde{C}_k^{\ref{EstiParametrixHmeasure}}) \beta(k\bgamma_a/2,\bgamma_a/2) \leq \tilde{C}(\gamma_a,\gamma_a')\left( (k-K) \frac{\kappa^k}{\bgamma_a^k(k!)^{\bgamma_a/2}} + \frac{(4C)^{k+1}}{\bgamma_a^k (k!)^{\bgamma_a/2}}\right),
\end{equation*}
for some positive real $\kappa \geq 4 C_{\ref{paramclassique}}$ and some positive constant $\tilde{C}(\gamma_a,\gamma_a')$.
\end{claim}

\begin{proof}
Let $K(\bgamma_a)$ be equal to $\lceil 2/\bgamma_a \rceil$, so that for all $k\geq K(\bgamma_a)$
$$k\bgamma_a/2-1\geq 0.$$ 
Then, let $k\geq K:= K(\bgamma_a) \vee K(\gamma_a)$. We know from the proof of Claim \ref{asymptobeta} that:
\begin{eqnarray*}
\beta(\bgamma_a k/2,\bgamma_a/2) \leq  \frac{4}{\bgamma_ak^{\bgamma_a/2}}.
\end{eqnarray*}

Hence, from definition of $C_k^{\ref{paramclassique}}$ we have that:
\begin{equation*}
\tilde{C}_{k+1}^{\ref{EstiParametrixHmeasure}}  \leq  C(K(\gamma_a))\frac{\tilde{C}\kappa^k}{\bgamma_a^{k}(k!)^{\bgamma_a/2}} +\tilde{C}_k^{\ref{EstiParametrixHmeasure}}C \frac{4}{\bgamma_ak^{\bgamma_a/2}},
\end{equation*}
so that
\begin{equation*}
\frac{\tilde{C}_{k+1}^{\ref{EstiParametrixHmeasure}}}{(k+1)^{\bgamma_a/2}} \leq C(K(\gamma_a))\frac{\tilde{C}\kappa^{k}}{\bgamma_a^{k}((k+1)!)^{\bgamma_a/2}} + \frac{\tilde{C}_k^{\ref{EstiParametrixHmeasure}}}{k^{\bgamma_a/2}} \frac{4C}{\bgamma_a(k+1)^{\bgamma_a/2}}.
\end{equation*}

Let us define
\begin{eqnarray*}
A_k = \frac{\tilde{C}_k^{\ref{EstiParametrixHmeasure}}}{k^{\gamma_a /2}},\quad M_{k+1} = \frac{C(K(\gamma_a)) \tilde{C}\kappa^k}{\bgamma_a^{k}((k+1)!)^{\bgamma_a/2}},\quad D_{k+1}=\frac{4C}{\bgamma_a ((k+1)!)^{\bgamma_a/2}}.
\end{eqnarray*}
Hence, we have that 
\begin{equation*}
A_{k+1} \leq M_{k+1} + A_k D_{k+1}.
\end{equation*}

Since $\kappa$ is such that $\kappa \geq 4 C_{\ref{paramclassique}}$ we obtain:
\begin{equation*}
M_k D_{k+1} = \frac{4 C_{\ref{paramclassique}}}{\gamma_a ((k+1)!)^{\bgamma_a/2}} \frac{C(K(\gamma_a))\tilde{C}\kappa^{k-1}}{\gamma_a^{k-1}(k!)^{\bgamma_a/2}} \leq \frac{\kappa^k \tilde{C}C(K(\bgamma_a))}{\bgamma_a^k((k+1)!)^{\bgamma_a/2}}=M_{k+1}.
\end{equation*}
Therefore, by induction
%
%
\begin{equation*}
A_{k+1}\leq (k-K) M_{k+1} + \prod_{i=K+1}^{k+1} \frac{4 C_{\ref{paramclassique}}}{\bgamma_a i^{\bgamma_a/2}},
\end{equation*}
which implies that 

\begin{equation*}
\frac{\tilde{C}_{k+1}}{(k+1)^{\bgamma_a/2}} \leq (k-K) \frac{C(K(\gamma_a))\tilde{C}\kappa^k}{\bgamma_a^k(k+1)!)^{\bgamma_a/2}} + \frac{(4C)^{k+1}}{\bgamma_a^k((k+1)!)^{\bgamma_a/2}}C(K),
\end{equation*}
where $C(K) = (4C)^{-K}\bgamma_a^K K!$. 
Then,

\begin{equation*}
\tilde{C}_{k+1}\leq  (k-K)C(K(\gamma_a))\tilde{C}  \frac{\kappa^k}{\bgamma_a^{k}(k!)^{\bgamma_a/2}} + C(K)\frac{(4C)^{k+1}}{\bgamma_a^{k}(k!)^{\bgamma_a/2}}.
\end{equation*}

\end{proof}

\bibliographystyle{alpha}
\bibliography{Bib_SMeasure}

\end{document}